\documentclass{amsart}

\usepackage{newtxtext, newtxmath}

\usepackage{geometry}
\usepackage{latexsym, bm}
\usepackage{amsmath, amssymb, amsfonts, amsthm, amscd}
\usepackage{cases}
\usepackage{enumerate, paralist}
\usepackage{appendix}

\usepackage[colorlinks=true]{hyperref}
\hypersetup{
	linkcolor=blue,
	anchorcolor=blue,
	citecolor=blue,
	filecolor=blue,
	urlcolor=blue,
	menucolor=blue,
}

\usepackage[msc-links, alphabetic, abbrev, nobysame]{amsrefs}

\usepackage{titletoc}

\usepackage{aliascnt}

\theoremstyle{theorem}
\newtheorem{thm}{Theorem}[section]

\newaliascnt{corollary}{thm}
\newtheorem{cor}[corollary]{Corollary}
\aliascntresetthe{corollary}

\newaliascnt{lemma}{thm}
\newtheorem{lem}[lemma]{Lemma}
\aliascntresetthe{lemma}

\newaliascnt{sublemma}{thm}

\aliascntresetthe{sublemma}

\theoremstyle{definition}
\newaliascnt{definition}{thm}
\newtheorem{defi}[definition]{Definition}
\aliascntresetthe{definition}

\newaliascnt{example}{thm}
\newtheorem{exa}[example]{Example}
\aliascntresetthe{example}

\newaliascnt{remark}{thm}
\newtheorem{rmk}[remark]{Remark}
\aliascntresetthe{remark}

\newcommand{\supp}{\mathrm{ supp }}

\newcommand{\id}{\mathrm{id}}

\allowdisplaybreaks

\numberwithin{equation}{section}

\begin{document}

\title{A note on Machado--Bishop theorem in weighted spaces with applications}

\author{Deliang Chen}
\address{School of Mathematical Sciences, Shanghai Jiao Tong University, Shanghai 200240, People's Republic of China}
\email{chernde@sjtu.edu.cn}



\begin{abstract}
A unified version of Machado--Bishop theorem in weighted spaces is given. A number of applications illustrate its importance.
\end{abstract}

\maketitle

\titlecontents{section}[0pt]{\vspace{0\baselineskip}\bfseries}
{\thecontentslabel\quad}{}
{\hspace{0em}\titlerule*[10pt]{$\cdot$}\contentspage}

\titlecontents{subsection}[1em]{\vspace{0\baselineskip}}
{\thecontentslabel\quad}{}
{\hspace{0em}\titlerule*[10pt]{$\cdot$}\contentspage}

\titlecontents{subsubsection}[2em]{\vspace{0\baselineskip}}
{\thecontentslabel\quad}{}
{\hspace{0em}\titlerule*[10pt]{$\cdot$}\contentspage}

\setcounter{tocdepth}{2}

\section{Introduction}

The Stone--Weierstrass theorem is one of the celebrated results in modern abstract analysis. A number of extensions were made by many authors. Among them, E. Bishop \cite{Bis61} gave an important observation that in order to verify $ g \in A $ where $ A $ is a closed unital subalgebra of $ C(\Omega) $ and $ \Omega $ is compact, one only considers $ g|_{E} \in A|_{E} $ for some simple subsets $ E \subset \Omega $ named the \emph{maximal $ A $-antisymmetric sets}; the original condition in the Stone--Weierstrass theorem in fact implies such $ E $ are single points. Later, S. Machado \cite{Mac77} formulated a quantitative version (i.e. a \emph{strong} version) of this theorem even in the vector-valued case. Other versions of this theorem were investigated by e.g. I. Glicksberg \cite{Gli63} (Bishop theorem for the strict topology), R. I. Jewett \cite{Jew63}, J. B. Prolla \cite{Pro88, Pro94}, and 
G. Paltineanu and I. Bucur \cite{PB17} (for $ C(\Omega, [0,1]) $ and $ C(\Omega, \mathbb{R}_+) $), and J. B. Prolla and S. Machado \cite{PM82} and M. S. Kashimoto \cite{Kas14} (for set-valued mappings), etc.

In his seminal article \cite{Nac65}, L. Nachbin introduced the weighted space to deal with the uniform approximation over a non-compact space. The author in that paper gave a description of the closure of some modules in weighted spaces based on the Stone--Weierstrass theorem, i.e., the so called weighted approximation problem; a more precise statement is to give some sufficient conditions for ``\emph{localization}'' (see e.g. \autoref{cor:localization}).
Generalizations of the Stone--Weierstrass theorem in weighted spaces in some sense have its advantage. For instance, by choosing different weighted families (i.e. Nachbin families), the topologies in weighted spaces include the point-open topology, the compact-open topology, the strict topology and the uniform topology (see \autoref{exa:spaces}); the weighted approximation problem in the so-called bounded case is related with the Stone--Weierstrass theorem which again turns out to be an important link with the Bernstein problem in the vector-valued case (see e.g. \autoref{exa:ber}); for more motivations, see \cite{Nac65}. J. B. Prolla \cite{Pro71} and W. H. Summers \cite{Sum71} gave a Bishop's version of the Stone--Weierstrass theorem in weighted spaces for locally compact spaces and for completely regular spaces, respectively; see also \cite{PM73} for the associated results. We refer the readers to see the more generalized results and important applications on this subject in the monograph \cite{Pro77}.

Meanwhile, lots of elementary proofs of the Stone--Weierstrass type of theorems were found; the above references we listed certainly not only gave generalized results but also provided simplified proofs. Here, the works of B. Brosowski and F. Deutsch (see \cite{BD81}) and T. J. Ransford (see \cite{Ran84}) need to be mentioned; their ingenious argument with minor changes in fact implies a more general version of Machado theorem which we give in this paper (see \autoref{thm:A}).

The present paper is to provide such a Machado theorem in the weighted spaces which recovers many known results we mentioned before; the statement is given in \autoref{sec:main}. Some new aspects of this theorem are listed in \autoref{rmk:new}. A number of applications in a more general setting based on this theorem are given in \autoref{sec:application} to illustrate the importance of this theorem.
\begin{enumerate}[(a)]
	\item A distance formula (see \autoref{thm:dis}) due to Burlando \cite{Bur05} is obtained.

	\item A version of the Stone--Weierstrass theorem (see \autoref{thm:nonseperate}) due to Timofte \cite[Theorem 7]{Tim05} is given; as corollaries, a general localization theorem (see \autoref{cor:localization}) due to Nachbin \cite[Theorem 2]{Nac65}, Prolla \cite{Pro71} and Summers \cite{Sum71}, and particularly a localization theorem for $ C(\Omega, [0,1]) $ (see \autoref{cor:C01}) due to von Neumann (see also \cite{Jew63}) are obtained. Also, a more classical form of the Stone--Weierstrass theorem (see \autoref{thm:SW}) is given for ease of application (see \autoref{sub:interpolation}).

	\item Some special forms of polynomial algebras are discussed in \autoref{sub:polynomial}; in particular, a version of the Bernstein problem in the vector-valued case (see \autoref{exa:ber}) is obtained.

	\item In \autoref{sub:interpolation}, we also give a brief discussion about the approximation problem with interpolation; a special case was studied in \cite[Section 4]{Pro94} and \cite{PK02}. Here we talk about a situation with more constraints in the sense of \cite{Tim05} (i.e. a good control on approximation's range); see \autoref{thm:ai0} and \autoref{cor:ai1} which generalize \cite[Theorem 16]{Tim05}. Finally, we pay attention on a concrete problem, i.e., to determine the closure of an operator's domain, which often occurs in the operator semigroup theory.
\end{enumerate}

\section{A generalized Machado--Bishop theorem}\label{sec:main}

Throughout this paper, the Hausdorff property of a topological space is not assumed unless where mentioned. We use the following notations:

$ \Omega $: a topological space (which might be not Hausdorff);

$ X $: a locally convex topological vector space with a family of seminorms denoted by $ \mathcal{A} $;

$ C(\Omega, X) $: the topological space consisting of all continuous functions of $ \Omega \to X $;

$ C_0(\Omega, X) $: the topological space consisting of all $ f \in C(\Omega, X) $ such that for all $ p \in \mathcal{A} $ and $ \epsilon > 0 $, $ \{ x \in \Omega: p(f(x)) \geq \epsilon \} $ is a compact subset of $ \Omega $;

$ \supp f $: the support of the map $ f: \Omega \to X $ defined by $ \supp f = \overline{f^{-1}(X \setminus \{0\})} $;

$ f|_S $: the restriction of the map $ f: \Omega \to X $ to $ S $ ($ \subset \Omega $);

$ V $: a Nachbin family, i.e., a set of upper semicontinuous functions of $ \Omega \to \mathbb{R}_+ $, whose elements are called \emph{weights}; see e.g. \cite{Nac65,Pro71, PM73, Pro77}.

Let $ W $ be a subset of $ C(\Omega, X) $. $ CV_0 W $ denotes the \emph{weighted space} of $ W $ (with respect to the Nachbin family $ V $), that is, $ f \in CV_0 W $ if and only if $ f \in W $ and for all $ v \in V $, $ p \in \mathcal{A} $, and $ \epsilon > 0 $, $ \{ x \in \Omega: v(x) p(f(x)) \geq \epsilon \} $ is a compact subset of $ \Omega $. For $ f \in CV_0 W $, $ F \subset \Omega $, define
\[
|f|_{v, p, F} = \sup \{ v(x)p(f(x)): x \in F \}, ~|f|_{v, p} = |f|_{v, p, \Omega},
\]
and for $ B \subset CV_0 W $,
\[
d_{v, p, F} (f, B) = \inf \{ |f - g|_{v, p, F}: g \in B \}.
\]
For brevity, if $ V, \mathcal{A} $ are the single point sets, then we write $ d_{F} (f, B) $ (or $ d_{v, F} (f, B), d_{p, F} (f, B) $) instead of $ d_{v, p, F} (f, B) $; if $ W = C(\Omega, X) $, then we write $ CV_0(\Omega, X) $ instead of $ CV_0C(\Omega, X) $.

Consider the following two additional conditions about $ V $:
\begin{enumerate}[(a)]
	\item for all $ v_1, v_2 \in V $, there are $ \lambda \geq 0 $ and $ w \in V $ such that $ v_1 \leq \lambda w $, $ v_2 \leq \lambda w $ (point-wise);
	\item for each $ x \in \Omega $, there is $ v \in V $ such that $ v(x) > 0 $.
\end{enumerate}
If $ CV_0 W $ is a linear subspace of $ C(\Omega, X) $, then the topological vector space $ CV_0 W $ is locally convex if condition (a) is satisfied and Hausdorff if condition (b) holds, where a local base at $ 0 $ can be given by $ \{ f \in CV_0 W : |f|_{v, p} < \epsilon \} $, $ v \in V $, $ p \in \mathcal{A} $, and $ \epsilon > 0 $.

Some examples of weighted spaces are the following. Let $ \chi_{K} $ be the characteristic function of $ K $, i.e., $ \chi_{K} (x) = 1 $ if $ x \in K $, and $ \chi_{K} (x) = 0 $ otherwise.
\begin{exa}\label{exa:spaces}
	\begin{enumerate}[(a)]
		\item Set $ V = \{ \chi_{K}: K \text{ is a finite set of } \Omega \} $, then $ CV_0 (\Omega, X) = C(\Omega, X) $, and the topology in this case is the point-open topology of $ C(\Omega, X) $.

		\item \label{exa:co} Set $ V = \{ \chi_{K}: K \text{ is a compact set of } \Omega \} $, then $ CV_0 (\Omega, X) = C(\Omega, X) $, and the topology in this case is the compact-open topology of $ C(\Omega, X) $.

		\item \label{exa:uniform} Set $ V = \{ 1 \} $, then $ CV_0 (\Omega, X) = C_0(\Omega, X) $. In this case, the topology is the uniform topology, i.e., $ f_{\alpha} \to 0 $ if and only if for all $ p \in \mathcal{A} $, $ p(f_{\alpha}(x)) \to 0 $ uniformly for $ x \in \Omega $.

		\item \label{exa:strict} $ V = C_0(\Omega, \mathbb{R}_+) $ and $ W = C_b(\Omega, X) $ (i.e. all the bounded continuous functions of $ \Omega \to X $), then $ CV_0 W = C_b(\Omega, X) $ and its topology is the \emph{strict topology}; see e.g. \cite{Buc58, Gli63, Sum71}.
	\end{enumerate}
\end{exa}

A modified conception of multiplier is given in the following which is sufficient for us (see also \cite{FdLP84, PK02}).

\begin{defi}[multiplier]
	For $ W \subset C(\Omega, X) $ and $ \varphi \in C(\Omega, \mathbb{C}) $, we say $ \varphi $ is a \emph{multiplier} of $ W $ if $ \varphi f + (1 - \varphi) g \in W $ for all $ f, g \in W $. A subset $ A \subset C(\Omega, \mathbb{C}) $ is called a \emph{multiplier} of $ W $ if for every $ \varphi \in A $, $ \varphi $ is a multiplier of $ W $.
\end{defi}

For example, if $ W \pm W \subset W $ and $ A W \subset W $, then $ A $ is a multiplier of $ W $ (which is frequently used in this paper). Indeed, for $ \varphi \in A $ and $ g, h \in W $, we have $ \varphi g \in W $ and $ (1 - \varphi) h = h - \varphi h \in W $ and hence $ \varphi g + (1 - \varphi) h \in W $. An elementary fact about multiplier is the following.
\begin{lem}\label{lem:mult}
	If $ \varphi $ is a multiplier of $ W $, so are $ \varphi^n $, $ 1 - \varphi^n $, and particularly $ (1 - \varphi^n)^m $ where $ n, m \in \mathbb{N}_+ $.
\end{lem}
\begin{proof}
	This is easy. For example, take $ n = 2 $. For all $ f, g \in W $, since $ \varphi f + (1 - \varphi) g \in W $, we have
	\[
	\varphi^2 f + (1 - \varphi^2) g = \varphi (\varphi f + (1 - \varphi) g) + (1 - \varphi) g \in W.
	\]
	The proof is complete.
\end{proof}

The following condition was used in \cite{Nac65, Pro71, Sum71}.
\begin{defi}[bounded condition]\label{def:bounded}
	We say a subset $ A $ of $ C(\Omega, \mathbb{C}) $ satisfies the \emph{bounded condition} (with respect to $ V $), if every $ \varphi \in A $ is bounded on the support of every $ v \in V $.
\end{defi}
For example, if $ A \subset C_{b}(\Omega, \mathbb{C}) $ or all the weights $ v \in V $ have compact supports, then $ A $ satisfies the bounded condition with respect to $ V $.

The following definition of antisymmetric set seems new.

\begin{defi}[antisymmetric set]\label{def:anti}
	\begin{enumerate}[(a)]
		\item For $ A \subset C(\Omega, \mathbb{C}) $, a subset $ S \subset \Omega $ is called an \emph{$ A $-antisymmetric set} if $ f \in A $ such that $ f: S \to [0,1] $, then $ f|_{S} $ is a constant.
		\item More generally, for $ v \in V $, a subset $ S \subset \supp v $ is called an \emph{$ (A, v) $-antisymmetric set} if $ f \in A $ such that $ f: S \to [0,1] $, then $ f|_{S} $ is a constant.
	\end{enumerate}
\end{defi}

\begin{rmk}\label{rmk:anti}
	\begin{enumerate}[(a)]
		\item If $ A $ is a subalgebra of $ C_b(\Omega, \mathbb{C}) $, then the definition of $ A $-antisymmetric set is the same as the classical one, i.e., $ S \subset \Omega $ is an $ A $-antisymmetric set if $ f \in A $ and $ f|_{S} $ is real then $ f|_{S} $ is constant (see e.g. \cite{Bis61}). We show this as follows.
		\begin{enumerate}[$ \bullet $]
			\item If $ S $ is an $ A $-antisymmetric set in the sense of \autoref{def:anti} (a) and if $ f \in A $ such that $ f|_{S} $ is real, then we can assume $ |f| < n $, and thus $ \frac{f^2}{n^2}: S \to [0,1] $, which yields $ f^2 $ is a constant on $ S $ (as $ \frac{f^2}{n^2} \in A $). We show $ f $ is constant on $ S $. If $ f^{2}|_{S} = 0 $, then it is true. So, without loss of generality, suppose $ f^{2}|_{S} = 1 $. Then $ \frac{f + f^{2}}{2} \in A $ (as $ A $ is a subalgebra) and $ \frac{f + f^{2}}{2}: S \to [0,1] $, yielding $ \frac{f + 1}{2}|_{S} = \frac{f + f^{2}}{2}|_{S} $ is a constant and consequently so is $ f|_{S} $.
		\end{enumerate}
	
		Note that in the classical definition of antisymmetric set, one usually considers $ C(\Omega, \mathbb{C}) $ if $ \Omega $ is compact or $ C_0(\Omega, \mathbb{C}) $ if $ \Omega $ is locally compact (see e.g. \cite{Bis61, Gli63, Mac77}).
		\item Introducing the notion of $ (A, v) $-antisymmetric set is to recover the results obtained in \cite{Nac65, Pro71, Sum71}. Assume $ A $ satisfies the \emph{bounded condition} with respect to $ V $, which was used in those papers cited before. Now let $ A $ be a subalgebra of $ C(\Omega, \mathbb{C}) $. The definition of $ A $-antisymmetric set $ S $ given in \cite{Pro71} is that if $ f \in A $ such that $ f|_S $ is real then $ f|_S $ is a constant. As the results given in \cite{Pro71} were all under the assumption that $ A $ satisfies above bounded condition, if $ S $ is an $ (A, v) $-antisymmetric set in the sense of \autoref{def:anti}, then for $ f \in A $ such that $ f|_S $ is real, the same argument given in (a) shows that $ f|_S $ is a constant due to the boundedness of $ f|_{\supp v} $ and $ S \subset \supp v $.
		\item Since in many cases, $ A \subset C_{b}(\Omega, \mathbb{C}) $ (or even $ A \subset C(\Omega, [0,1]) $) or all the weights $ v \in V $ have compact supports, the definition of \emph{strong $ A $-antisymmetric set} $ S $ like that $ A_S = \{ f|_{S}: f\in A \} $ contains no nonconstant real functions is sufficient.
	\end{enumerate}
\end{rmk}

Under the above discussion, we have
\[
\text{strong $ A $-antisymmetric set}~ S \Rightarrow ~\text{$ A $-antisymmetric set}~ S \Rightarrow ~\text{$ (A,v) $-antisymmetric set}~ S \cap \supp v;
\]
and if $ A $ is a subalgebra of $ C_b(\Omega, \mathbb{C}) $ or $ A  \subset C(\Omega, [0,1]) $, then
\[
\text{strong $ A $-antisymmetric set}~ S \Leftrightarrow ~\text{$ A $-antisymmetric set}~ S;
\]
and if $ A $ is a subalgebra of $ C(\Omega, \mathbb{C}) $ satisfying the bounded condition with respect to $ V $, then for $ v \in V $ and $ S \subset \supp v $,
\[
\text{$ (A,v) $-antisymmetric set}~ S \Rightarrow ~\text{strong $ A $-antisymmetric set}~ S.
\]

For $ W \subset C(\Omega, X) $, write
\[
[x] \triangleq [x]_{W} = \{ y \in \Omega: f(y) = f(x), \forall f \in W \}.
\]
$ \{[x]: x \in \Omega \} $ defines an equivalence relation in $ \Omega $ (with respect to $ W $), denoted by $ \rho_{W} $; the equivalence classes are $ [x]_{W} $, $ x \in \Omega $. For brevity, for $ S \subset \Omega $, the notation $ S \subset \rho_{W} $ means $ f(x) = f(y) $ for every $ x, y \in S $ and every $ f \in W $. We also write $ \rho_{W} \subset \rho_{W_1} $ if $ S \subset \rho_{W} \Rightarrow S \subset \rho_{W_1} $ (see also \cite{Tim05}). For $ S \subset \Omega $, $ t \in \Omega $, let
\[
W_{S} = \{ f|_{S}: f\in W \} \subset C(S, X), ~ W(t) = \{ f(t): f \in W \}.
\]
We say $ W $ \emph{separates} $ S $ if for every $ x, y \in S $ such that $ x \neq y $ there is $ f \in W $ such that $ f(x) \neq f(y) $.

\begin{exa}\label{exa:anti}
	Let $ A \subset C(\Omega, \mathbb{C}) $.
	\begin{enumerate}[(a)]
		\item If $ A \subset C(\Omega, [0,1]) $, then $ [x]_{A} $, $ x \in \Omega $ are all the maximal $ A $-antisymmetric sets. Note that if, in addition, $ A $ separates $ \Omega $, then $ [x]_{A} = \{x\} $.
		\item If for every $ t \in \Omega $, $ A(t) \cap [0,1] = \emptyset $, then every subset of $ \Omega $ is an $ A $-antisymmetric set.
		\item If $ A $ is a subalgebra of $ C_b(\Omega, \mathbb{C}) $ such that $ A $ is self-adjoint (i.e. if $ f \in A $ then $ \overline{f} \in A $), then $ [x]_{A} $, $ x \in \Omega $ are all the maximal $ A $-antisymmetric sets. (Note that $ f + \overline{f} $ and $ i f + \overline{ if } $ are real.)
		\item \label{symmetric} If $ A $ is a subalgebra of $ C_b(\Omega, \mathbb{C}) $ such that $ A $ is self-adjoint and separates $ \Omega \setminus S $ where $ S $ is a closed subset of $ \Omega $ such that $ A_S = 0 $, then the all maximal $ A $-antisymmetric sets are $ \{x\} $ ($ x \in \Omega \setminus S $) and $ S $, or $ \{x\} $ ($ x \in \Omega \setminus \{S \cup \{x_0\}\} $) and $ S \cup \{x_0\} $ for some $ x_0 \in \Omega \setminus S $.
		\item \label{ee} Let $ A $ denote all the polynomials in $ \mathbb{D} = \{ z \in \mathbb{C}: |z| \leq 1 \} $. Although $ A $ separates $ \mathbb{D} $ (and so $ [x]_{A} = \{x\} $), the maximal $ A $-antisymmetric set is only $ \mathbb{D} $.
	\end{enumerate}
\end{exa}

Under the above preliminaries, we are in a position to state our generalized Machado--Bishop theorem; the original version of the Machado theorem was in \cite{Mac77} which can be considered as a \emph{strong} version of the Stone--Weierstrass theorem.

\begin{thm}[Machado Theorem]\label{thm:A}
	Let $ A $ be a non-empty subset of $ C(\Omega, \mathbb{C}) $ and $ W $ a topological vector subspace of $ C(\Omega, X) $. Take a subset $ W_0 $ of $ CV_0 W $. Assume $ A $ is a multiplier of $ W_0 $. Then for every $ f \in CV_0 W $, $ v \in V $ and $ p \in \mathcal{A} $, there is an $ (A, v) $-antisymmetric set $ S $ such that
	\[
	d_{v, p, S}(f, W_0) = d_{v, p, \Omega}(f, W_0).
	\]
\end{thm}

\begin{proof}
	Here, the argument due to Brosowski--Deutsch \cite{BD81} and Ransford \cite{Ran84} can be applied. We give the details as follows. Let $ a = d_{v, p, \Omega}(f, W_0) $. If $ a = 0 $, then the proof is finished. So assume $ a > 0 $. Set
	\[
	\mathcal{F} \triangleq \{F \subset \supp v: F ~\text{is a closed subset of $ \supp v $ and}~ d_{v, p, F}(f, W_0) = a \}.
	\]
	Obviously, $\mathcal{F} \neq \emptyset$ as $\supp v \in \mathcal{F}$. Given a natural partial order $\supset$ in $\mathcal{F}$, now $(\mathcal{F}, \supset)$ is a partially ordered set. We claim if $\mathcal{E}$ is a totally ordered subset of $\mathcal{F}$, then $F_0 = \bigcap\{ F: F \in \mathcal{E} \} \in \mathcal{F}$. Indeed, for any $ 0 < \varepsilon < a / 2 $, if $g \in W_0$ and $F \in \mathcal{E}$, then
	\[
	\Theta_F \triangleq \{ x \in F: v(x)p(f(x) - g(x)) \geq a - \varepsilon \},
	\]
	is closed (due to $ v, p, f, g $ are upper semicontinuous) and compact (due to $ f - g \in CV_0 W $ and $ a - \varepsilon > 0 $); in addition, $ \Theta_F \neq \emptyset $ as $ d_{v, p, F}(f, W_0) = a $. Therefore, $\bigcap\limits_{F \in \mathcal{E}}\{ \Theta_F \}$ is non-empty and closed. That is for all $g \in W_0$,
	\[
	\{ x \in F_0: v(x)p(f(x) - g(x)) \geq a - \varepsilon \} \neq \emptyset,
	\]
	and so $d_{v, p, F_0}(f, W_0) \geq a - \varepsilon$. As $ \varepsilon $ is arbitrarily small, one gets $F_0 \in \mathcal{F}$. By the Hausdorff's maximality theorem (see e.g. \cite{Rud87}), $\mathcal{F}$ has a maximal element $S \in \mathcal{F}$; evidently $ S $ is closed. In the following we show $S$ is an $(A, v)$-antisymmetric set.

	Otherwise, there is $h \in A$ such that $h: S \to [0,1]$ but it is not a constant. Let
	\[
	b = \inf\limits_{t \in S}h(t), ~c = \sup\limits_{t \in S}h(t).
	\]
	Then $ b < c $. Take integrals $c_1, c_2, c_3$ such that $0 \leq b<1/c_1<1/c_2<1/c_3<c \leq 1$. Let
	\[
	Y \triangleq \{ t \in S : 0\leq h(t) \leq 1/c_3 \}, ~Z \triangleq \{ t \in S : 1/c_1 \leq h(t) \leq 1 \}.
	\]
	Then $Y, Z$ are closed and proper subsets of $S$. By the maximal property of $S$ in $ \mathcal{F} $, there are $g_Y, g_Z \in W_0$ such that
	\[
	|f-g_Y|_{v, p, Y} < a, ~|f-g_Z|_{v, p, Z} < a.
	\]
	Take
	\[
	h_n = (1-h^n)^{c_2^n}, ~ g_n = h_n\cdot g_Y + (1-h_n)\cdot g_Z, ~ n \in \mathbb{N}.
	\]
	As $ h $ is a multiplier of $ W_0 $, we have $ g_n \in W_0 $ by \autoref{lem:mult}. Also note that $ |f - g_n|_{v, p, Y \cap Z} < a $ for $ g_n $ is a convex combination $g_Y, g_Z$ and $ 0 \leq h_n|_{S} \leq 1 $.

	For $t \in Y \backslash Z$, we have $0\leq h(t) < 1/c_1$, and
	$$
	h_n(t) = (1-h^n(t))^{c_2^n} \geq 1 - c_2^nh^n(t) \geq 1 - (c_2/c_1)^n;
	$$
	and for $t \in Z \backslash Y$, we have $1/c_3 < h(t) \leq 1$, and
	\[
	h_n(t) = (1-h^n(t))^{c_2^n} \leq (1+h(t))^{-c_2^n} \leq (c^n_2h^n(t))^{-1} \leq (c_3/c_2)^n.
	\]
	(Here, the following simple inequalities are used: if $|x|\leq1, ~p \geq 1$, then $(1-x)^p \geq 1-px$ and $(1+x)^p \geq px$.) Hence, $g_n \rightrightarrows g_Y$ in $Y \backslash Z$, and $g_n \rightrightarrows g_Z$ in $Z \backslash Y$. Now, we get
	\[
	|f-g_n|_{v, p, S} = \max\{ |f-g_n|_{v, p, Y \backslash Z}, |f-g_n|_{v, p, Z \backslash Y}, |f-g_n|_{v, p, Z \cap Y} \} < a,
	\]
	for sufficiently large $n$. Consequently, $d_{v, p, S}(f, W_0) < a$, which is a contradiction as $S \in \mathcal{F}$. The proof is complete.
\end{proof}

Let us give a simple corollary of \autoref{thm:A} due to Bishop \cite{Bis61} and Prolla \cite{Pro71, Pro88}.

\begin{cor}\label{cor:A}
	Let $ W $ be a topological vector subspace of $ C(\Omega, X) $ and $ W_0 \subset CV_0W $. Take $ f \in CV_0W $.
	\begin{enumerate}[(a)]
		\item \label{cor:Aa} (See Bishop \cite{Bis61} and Prolla \cite{Pro71}) Suppose $ A $ is a subalgebra of $ C(\Omega, \mathbb{C}) $ and $ W_0 $ is a closed linear subspace of $ CV_0W $ and an $ A $-module (i.e. $ A W_0 \subset W_0 $). If for every $ v \in V $ and for every $ (A, v) $-antisymmetric set $ S \subset \supp v $, there is $ g \in W_0 $ such that $ f|_{S} = g|_{S} $, then $ f \in W_0 $.

		\item \label{cor:Ab} (See Prolla \cite{Pro88}) Let $ A \subset C(\Omega, [0,1]) $ such that it is a multiplier of $ W_0 $, then for every $ v \in V $ and $ p \in \mathcal{A} $, there is $ x \in \Omega $ such that
		\[
		d_{v, p, [x]_{A}}(f, W_0) = d_{v, p, \Omega}(f, W_0).
		\]
	\end{enumerate}
\end{cor}

\begin{proof}
	(a) Note that if $ A $ is a subalgebra and $ W_0 $ is an $ A $-module, then $ A $ is a multiplier of $ W_0 $.

	(b) Note that $ [x]_{A} $, $ x \in \Omega $ are all the maximal $ A $-antisymmetric sets.
\end{proof}

\begin{rmk}[What is new?]\label{rmk:new}
	\begin{enumerate}[(a)]
		\item To our knowledge, this version of the Machado theorem seems new in weighted spaces. We will use this theorem to reprove some results obtained in \cite{Bur05, Tim05} with some generalizations.

		\item This is a version that without assuming $ A $ is a subalgebra or $ W_0 $ is a linear space; what we need is that $ A $ is a multiplier of $ W_0 $ which was also noted in e.g. \cite{Pro88, Pro94}. In some cases, we will take $ A = C(\Omega, [0,1]) $ which is not a subalgebra; see e.g. \autoref{thm:dis} and \autoref{thm:ai0}. For another application of this situation about the characterization of convex cones in continuous spaces, see also \cite{Pro88, Pro94}.

		\item Even in the case that $ A $ is a subalgebra, in general, $ A $ does not need to contain constants; in some cases, this is also important for applications (see e.g. \autoref{thm:SW} and \autoref{sub:interpolation}). See also \autoref{cor:C01} in the same spirit. We learned this fact from \cite[p. 403]{Rud91}. As a matter of fact, in \cite{Pro71}, this was not assumed as well.

		\item In general, we do not assume $ \Omega $ or $ X $ is Hausdorff (as well as $ CV_0W $); this was also done in \cite{Bur05, Tim05} where $ \Omega $ is not assumed to be Hausdorff. But, in some concrete applications, we usually need one of $ \Omega, X $ fulfills the Hausdorff property; see \autoref{sec:application} for details.

		\item Based on our definition of $ (A, v) $-antisymmetric set, \autoref{thm:A} recovers the complex case which cannot be done in \cite{Pro88, Pro94, PK02, Tim05} (as they chose $ A \subset C(\Omega, [0,1]) $; see also \autoref{exa:anti} \eqref{ee}). Also, this theorem contains the bounded case in the weighted approximation problem studied in \cite{Nac65, Pro71, Sum71}; see \autoref{rmk:anti}.
	\end{enumerate}
\end{rmk}

\section{some applications} \label{sec:application}

\subsection{A distance formula}

Let $ C_0(\Omega, X) $ be endowed with the uniform topology (see \autoref{exa:spaces} \eqref{exa:uniform}). Then we have the following which generalizes a partial result of Burlando \cite{Bur05}.
\begin{thm}[See Burlando {\cite[Theorem 3.11]{Bur05}}] \label{thm:dis}
	Assume the following non-degenerate condition holds: for every $ t \in \Omega $, there is $ \varphi \in C_0(\Omega, [0,1]) $ such that $ \varphi(t) \neq 0 $. If $ X $ is Hausdorff and $ X_0 $ is a convex subset of $ X $, then for all $ p \in \mathcal{A} $ and $ f \in C_0(\Omega, X) $, we have
	\[\label{equ:a1}\tag{$ \divideontimes $}
	\max_{t \in \Omega} d_{p} (f(t), X_0) = d_{p, \Omega} (f, C_0(\Omega, X_0)),
	\]
	where $ d_{p} (f(t), X_0) \triangleq \inf \{ p(f(t) - x): x \in X_0 \} $.
\end{thm}

\begin{proof}
	Fix $ p \in \mathcal{A} $.
	As $ X_0 $ is convex, we see $ A = C(\Omega, [0,1]) $ is a multiplier of $ C_0(\Omega, X_0) $. Next we show if $ y \in [x]_{A} $, then $ f(x) = f(y) $. Since $ f \in C_0(\Omega, X) $, we have $ \sup\{ p(f(x)): x \in \Omega \} < \infty $. For $ x^* \in X^* $ (the dual space of $ X $), assume $  |\mathrm{Re} x^*f(t)| \leq d $ for some $ d > 0 $ depending on $ x^*, p $. Then for $ \phi(t) = (2d)^{-1} (\mathrm{Re} x^*f(t) + d ) $, one has $ \phi \in C(\Omega, [0,1]) $. Thus, $ \mathrm{Re} x^*f(x) = \mathrm{Re} x^*f(y) $. This shows that $ f(x) = f(y) $; note that $ X^* $ separates $ X $ for $ X $ is locally convex and Hausdorff (see \cite{Rud91}). By \autoref{cor:A} \eqref{cor:Ab}, we have
	\begin{align*}
	\sup_{t\in \Omega} d_{p} (f(t), X_0) & \leq d_{p, \Omega}(f, C_0(\Omega, X_0)) \\
	& = d_{p, [s]_{A}} (f, C_0(\Omega, X_0)) \quad \text{~for some $ s \in \Omega $}\\
	& = \inf\{ p(f(s), g(s)): g \in C_0(\Omega, X_0) \} \\
	& = \inf\{ p(f(s), z): z \in X_0 \} ~\text{(non-degenerate condition)} \\
	& = d_{p} (f(s), X_0).
	\end{align*}
	The proof is complete.
\end{proof}

\begin{rmk}
	Let us discuss the \emph{non-degenerate condition}: for every $ t \in \Omega $, there is $ \varphi \in C_0(\Omega, [0,1]) $ such that $ \varphi(t) \neq 0 $. Let $ A = C(\Omega, [0,1]) $.
	\begin{enumerate}[(a)]
		\item If $ \Omega $ is compact, then the non-degenerate condition is clearly true as in this case $ C_0(\Omega, [0,1]) = C(\Omega, [0,1]) $. In fact, this is the case discussed in \cite{Bur05}.
		\item When $ \Omega $ is Hausdorff, it is well known that the non-degenerate condition is equivalent to that $ \Omega $ is locally compact.
		\item If non-degenerate condition holds, then $ \Omega $ must be locally compact and for every $ x_0 \in \Omega $, there is a compact neighborhood $ U_{x_0} $ of $ x_0 $ such that $ \bigcup_{x \in U_{x_0}} [x]_{A} = U_{x_0} $, i.e., $ \bigcup_{x \in U_{x_0}} [x]_{A} = U_{x_0} = \{ x: \varphi(x) \geq \varphi(x_0) / 2 \} $ where $ \varphi \in C_0(\Omega, [0,1]) $ such that $ \varphi(x_0) \neq 0 $; particularly $ [x]_{A} $ is compact in $ \Omega $.

		Let us show the converse is also true. The equivalence classes $ [x]_{A} $, $ x \in \Omega $ define a quotient space $ \widetilde{\Omega} $ which is Hausdorff and locally compact and a quotient map $ \tau: x \to [x]_{A} $. The condition now means that $ \tau $ is proper, i.e., for every compact subset $ \widetilde{M} \subset \widetilde{\Omega} $, $ \tau^{-1}(\widetilde{M}) $ is compact. More precisely, let $ \{x_{t}\} $ be a net in $ \tau^{-1}(\widetilde{M}) $, then there is a subnet $ \{\tau(x_{t_{\delta}})\} $ such that $ \tau(x_{t_{\delta}}) \to [\hat{x}]_{A} \in \widetilde{M} $. Since $ U_{\hat{x}} = \tau^{-1} \tau(U_{\hat{x}}) $, and $ \tau(U_{\hat{x}}) $ is a neighborhood of $ [\hat{x}]_{A} $, one gets $ x_{t_{\delta}} \in U_{\hat{x}} $ for $ \delta \geq \delta_0 $. Due to the compactness of $ U_{\hat{x}} $, we obtain $ \{x_{t_{\delta}}\} $ has a convergent subnet. For every $ t \in {\Omega} $, now we have $ \widetilde{\varphi} \in C_0(\widetilde{\Omega}, [0,1]) $ such that $ \widetilde{\varphi}([t]_{A}) \neq 0 $. Set $ \varphi = \widetilde{\varphi} \circ \tau $, which is the desired function.

		\item We mention that in \cite{Bur05}, the author also proved that when $ \Omega $ is compact and $ \widetilde{\Omega} $ is totally disconnected, the equality \eqref{equ:a1} holds without assuming $ X_0 $ is convex. The examples given in \cite{Bur05} also yield that when $ \Omega $ is compact, the condition that $ \widetilde{\Omega} $ is totally disconnected or $ X_0 $ is convex in some sense is optimal.
	\end{enumerate}
\end{rmk}

\subsection{Stone--Weierstrass type of results}

\begin{thm}[See Timofte {\cite[Theorem 7]{Tim05}}] \label{thm:nonseperate}
	Let $ A $ be a non-empty subset of $ C(\Omega, \mathbb{C}) $ and $ W $ a topological vector subspace of $ C(\Omega, X) $.
	Take a subset $ W_0 $ of $ CV_0 W $. Assume $ A $ is a multiplier of $ W_0 $ and $ S \subset \rho_{W_0} $ for every $ (A, v) $-antisymmetric set $ S $ ($ v \in V $). Let
	\[
	\mathcal{W} = \{ f \in CV_0W: \rho_{W_0} \subset \rho_{\{f\}}, f(t) \in \overline{W_0(t)} ~\text{for all}~ t \in \Omega \},
	\]
	and
	\begin{multline*}
	\mathcal{W}_0 = \{ f \in CV_0W: S \subset \rho_{\{f\}} ~\text{for all $ (A, v) $-antisymmetric set $ S $ and all}~ v \in V,\\ f(t) \in \overline{W_0(t)} ~\text{for all}~ t \in \Omega \}.
	\end{multline*}
	Then $ \overline{W_0} = \mathcal{W}_0 = \mathcal{W} $ in $ CV_0 W $.
\end{thm}

\begin{proof}
	Clearly, $ \overline{W_0} \subset \mathcal{W} \subset \mathcal{W}_0 $. Let $ f \in \mathcal{W}_0 $. For every $ v \in V $ and $ p \in \mathcal{A} $, by \autoref{thm:A}, there is an $ (A, v) $-antisymmetric set $ S $ such that $ d_{v, p, S}(f, W_0) = d_{v, p, \Omega}(f, W_0) $. Take $ t_0 \in S $. As $ S \subset \rho_{W_0} $ and $ S \subset \rho_{\{f\}} $, we have $ g(t) = g(t_0) $ and $ f(t) = f(t_0) $ for any $ t \in S $ and $ g \in W_0 $. We can assume $ v|_S $ is bounded; indeed, if for any $ g \in W_0 $, $ p(f(t_0) - g(t_0)) = 0 $, then $ d_{v, p, S}(f, W_0) = 0 $ and the result follows, else there is $ g_0 \in W_0 $ such that $ p(f(t_0) - g_0(t_0)) \neq 0 $, then due to $ \sup_{t \in \Omega} |v(t)p(f(t) - g_0(t))| < C_1 $ for some $ C_1 $, we have $ \sup_{t \in S}|v(t)| \leq C_1 |p(f(t_0) - g_0(t_0))|^{-1} < \infty $.
	For any $ \varepsilon > 0 $, choose $ t'_{g} \in S $ such that $ v(t'_{g})p(f(t'_{g}) - g(t'_{g})) \geq \sup_{t \in S} v(t)p(f(t) - g(t)) - \varepsilon $. Now by $ f(t_0) \in \overline{W_0(t_0)} $, we have
	\[
	\inf_{g \in W_0} v(t'_{g})p(f(t'_{g}) - g(t'_{g})) = \inf_{g \in W_0} v(t'_{g})p(f(t_0) - g(t_0)) \leq \sup_{t \in S}|v(t)| \inf_{g \in W_0} p(f(t_0) - g(t_0)) = 0,
	\]
	i.e., $ \inf_{g \in W_0} \sup_{t \in S} v(t)p(f(t) - g(t)) \leq \varepsilon $, which yields $ d_{v, p, S}(f, W_0) = 0 $. Therefore, $ d_{v, p, \Omega}(f, W_0) = 0 $ for every $ v \in V $ and $ p \in \mathcal{A} $, that is, $ f \in \overline{W_0} $. The proof is complete.
\end{proof}

\begin{rmk}
	In \cite{Tim05}, Timofte proved $ \overline{W_0} = \mathcal{W} $ (but see also \cite[Proposition 3]{Tim05}) in the case $ CV_0W $ is one of the spaces listed in \autoref{exa:spaces} (especially all the weights are bounded) and $ A \subset C(\Omega, [0,1]) $; note that now $ S \subset \rho_{W_0} $ for every $ A $-antisymmetric set $ S $ is equivalent to  $ \rho_{A} \subset \rho_{W_0} $.
\end{rmk}

In analogy with \cite[Definition 3.5]{Pro71}, we say a non-empty subset $ A $ of $ C(\Omega, \mathbb{C}) $ is \emph{symmetric} if very maximal $ A $-antisymmetric set (or $ (A, v) $-antisymmetric set) reduces to a point (see e.g. \autoref{exa:anti} \eqref{symmetric}). The following result now becomes a simple consequence of \autoref{thm:nonseperate} which is a little bit more general than Nachbin \cite[Theorem 2]{Nac65}, Prolla \cite{Pro71} and Summers \cite{Sum71}; here also note that the \emph{bounded condition} of $ A $ (see \autoref{def:bounded}) in general is not needed (but see also the \autoref{def:anti} of antisymmetric set and \autoref{rmk:anti}).
\begin{cor}[Localization Theorem] \label{cor:localization}
	Let $ A, CV_0W, W_0 $ be as in \autoref{thm:nonseperate}. Assume $ A $ is a multiplier of $ W_0 $ and is symmetric. Then $ f \in \overline{W_0} $ if and only if $ f(t) \in \overline{W_0(t)} $ for every $ t \in \Omega $.
\end{cor}

The following corollary is essentially due to von Neumann (see also \cite{Jew63}) in the setting that $ \Omega $ is compact, $ 0, 1, c \in M $ for some constant $ 0 < c < 1 $ and $ S = \emptyset $.

\begin{cor}[Localization Theorem for {$ C(\Omega, [0,1]) $}]\label{cor:C01}
	Let $ M \subset C(\Omega, [0,1]) $ such that it is a multiplier of itself (i.e. for all $ \varphi \in M $, $ \varphi M + (1 - \varphi) M \subset M $) and $ S $ a closed subset of $ \Omega $. Assume that 
	\begin{enumerate}[(i)]
		\item $ M $ separates $ \Omega \setminus S $,
		
		\item for every $ t \in \Omega \setminus S $, $ \exists m_{i, t} \in M $, $ i = 1, 2, 3 $, such that $ m_{1, t}(t) = 0 $, $ m_{2, t}(t) = 1 $ and $ 0 < m_{3, t}(t) < 1 $,
		
		\item $ M|_S = \{ 0 \} $. 
	\end{enumerate}
	Then $ \overline{M} = \{ f \in C(\Omega, [0,1]): f|_{S} = 0 \} $ where the topology is taken as the compact-open topology or the strict topology (see \autoref{exa:spaces} \eqref{exa:co} \eqref{exa:strict}).
\end{cor}

\begin{proof}
	By the conditions (i) (ii) on $ M $, all the $ (M, v) $-antisymmetric sets are $ \{x\} $ ($ x \in \Omega \setminus S $) and $ S \cap \supp v $ ($ v \in V $) where $ V $ is taken as in \autoref{exa:spaces} \eqref{exa:co} or \eqref{exa:strict}. We claim that for $ t \in \Omega \setminus S $, $ \overline{M(t)} = [0,1] $. By \autoref{lem:mult}, we get for every $ n, k \in \mathbb{N}_+ $,
	\[
	(1 - (1 - m^{n}_{3,t})^{k}) m_{1, t} + (1 - m^{n}_{3,t})^{k} m_{2, t} \in M,
	\]
	and so $ (1 - m^{n}_{3,t}(t))^{k} \in M(t) $. Since $ 0 < m_{3,t}(t) < 1 $, one can easily see that for every $ \epsilon > 0 $ and $ a \in (0, 1) $, there are $ n, k \in \mathbb{N}_+ $ such that $ |a - (1 - m^{n}_{3,t}(t))^{k}| < \epsilon $ (see also \cite[Lemma 4.10]{PB17}). Therefore, $ \overline{M(t)} = [0,1] $. This shows $ \mathcal{W}_0 $ defined in \autoref{thm:nonseperate} is equal to $ \{ f \in C(\Omega, [0,1]): f|_{S} = 0 \} $; or use \autoref{thm:A} directly. The proof is complete.
\end{proof}

The following type of the Stone--Weierstrass theorem was reproved by many authors in different settings, see e.g. \cite{Gli63, Pro71, Tim05} etc; usually for the case $ 1 \in A $ and $ S = \emptyset $ (but see also \cite[Corollary 2.11]{PM73}).

\begin{thm}[Stone--Weierstrass Theorem]\label{thm:SW}
	Let $ A $ be a subalgebra of $ C(\Omega, \mathbb{C}) $ satisfying the bounded condition with respect to $ V $ (see \autoref{def:bounded}) and $ W $ a topological vector subspace of $ C(\Omega, X) $. Let $ S $ be a closed subset of $ \Omega $. Take a linear subspace $ G $ of $ CV_0W $ such that it is an $ A $-module (i.e. $ A G \subset G $) and $ G_{S} = 0 $. Assume the following conditions hold:
	\begin{enumerate}[(i)]
		\item $ A $ is self-adjoint (i.e. $ f \in A \Rightarrow \overline{f} \in A $);
		\item $ A $ separates $ \Omega \setminus S $;
		\item $ A_{S} = 0 $, and for every $ t \in \Omega \setminus S $ there is $ \varphi \in A $ such that $ \varphi(t) \neq 0 $;
		\item for every $ \varepsilon > 0 $, $ p \in \mathcal{A} $, and for every $ t \in \Omega \setminus S $, $ y \in X $, there exists $ f \in G $ such that $ p(f (t) - y) < \varepsilon $.
	\end{enumerate}
	Then $ \overline{G} = \{ f \in CV_0W: f|_{S} = 0 \} $ in $ CV_0W $ and $ \overline{A} = \{ \varphi \in C(\Omega, \mathbb{C}): \varphi|_{S} = 0 \} $ in the compact-open topology (see \autoref{exa:spaces} \eqref{exa:co}).
\end{thm}
\begin{proof}
	Since $ A $ is a subalgebra and satisfies the bounded condition with respect to $ V $, the conditions (i) (ii) (iii) now imply that all the $ (A, v) $-antisymmetric sets are $ \{x\} $ ($ x \in \Omega \setminus S $) and $ S \cap \supp v $ (where $ v \in V $); see e.g. \autoref{rmk:anti}. In particular, if $ C(\Omega, \mathbb{C}) $ is endowed with the compact-open topology, then due to $ A(t) = \mathbb{C} $ for $ t \in \Omega \setminus S $ and $ A_S = 0 $, we have $ \overline{A} = \{ \varphi \in C(\Omega, \mathbb{C}): \varphi|_{S} = 0 \} $ (by \autoref{thm:A} or \autoref{thm:nonseperate}). Similarly, as $ \overline{G(t)} = X $ for $ t \in \Omega \setminus S $ (by condition (iv)) and $ G_S = 0 $, one gets $ \overline{G} = \{ f \in CV_0W: f|_{S} = 0 \} $ in $ CV_0W $; note that $ A $ is a multiplier of $ G $. The proof is complete.
\end{proof}

\subsection{polynomial algebra}\label{sub:polynomial}

For an index $ I $ and a linear subspace $ Y $ of $ X $, we write $ P(Z) \in Y[Z], Z = (Z_i)_{i \in I} $, which means $ P(Z) $ is a finite sum of terms of the type $ x Z_{i_1}Z_{i_2}\cdots Z_{i_m} $, where $ i_1, i_2, \ldots, i_m \in I $ and $ x \in Y $; here we also \emph{assume} $ \{i_1, i_2, \ldots, i_m\} \neq \emptyset $. For $ A \subset C(\Omega, \mathbb{C}) $ and $ X_0 \subset X $, set
\[
A \otimes X_0 = \mathrm{span} \{ \varphi \otimes x: \varphi \in A, x \in X_0 \},
\]
where $ (\varphi \otimes x)(t) \triangleq \varphi(t)x $.

The following result is classical (see e.g. \cite[Theorem 6.1]{Bor76}) especially for the case $ \Omega $ is compact, $ X_0 = X $ and $ S = \emptyset $.

\begin{thm}[Grothendieck--Stone--Weierstrass Theorem]
	Let $ \{ u_{i} \}_{i \in I} \subset CV_0(\Omega, \mathbb{R}) $ such that it satisfies the bounded condition with respect to $ V $ (see \autoref{def:bounded}) and $ S $ a closed subset of $ \Omega $. Assume 
	\begin{enumerate}[(i)]
		\item $ \{ u_{i} \}_{i \in I} $ separates $ \Omega \setminus S $,
		\item $ u_{j}|_S = 0 $ ($ j \in I $), and
		\item for each $ t \in \Omega \setminus S $, there is $ j \in I $ such that $ u_{j}(t) \neq 0 $.
	\end{enumerate}
	Let $ X_0 $ be a dense linear subspace of $ X $. Then the following polynomial algebra
	\[
	\{ P((u_i)_{i \in I}): P(Z) \in X_0[Z], Z = (Z_i)_{i \in I} \}
	\]
	is dense in $ \{ f \in CV_0(\Omega, X): f|_{S} = 0 \} $ (in the topology of $ CV_0(\Omega, X) $).
\end{thm}

\begin{proof}
	Let $ A = <\{ u_{i} \}_{i \in I}> \triangleq \{ P((u_i)_{i \in I}): P(Z) \in \mathbb{C}[Z], Z = (Z_i)_{i \in I} \} $ (i.e. the self-adjoint algebra generated by $ \{ u_{i} \}_{i \in I} $); in general we do not know whether $ 1 \in A $. Clearly, $ A \subset CV_0(\Omega, \mathbb{C}) $ and satisfies the bounded condition with respect to $ V $. Note that $ A \otimes X_0 = \{ P((u_i)_{i \in I}): P(Z) \in X_0[Z], Z = (Z_i)_{i \in I} \} \subset CV_0(\Omega, X) $ and $ A (A \otimes X_0) \subset A \otimes X_0 $. Now, the result follows from \autoref{thm:SW}.
\end{proof}

From the above theorem, if $ A $ is a unital subalgebra of $ C(\Omega, \mathbb{C}) $ (resp. $ C_b(\Omega, \mathbb{C}) $) which has a real base such that it separates $ \Omega $, and if $ X_0 $ is a dense linear subspace of $ X $, then $ A \otimes X_0 $ is dense in $ C(\Omega, X) $ (resp. $ C_b(\Omega, X) $) in the compact-open topology (resp. in the strict topology).

\begin{exa}
	Let $ \Omega $ be compact. Assume $ \Omega $ or $ X $ is Hausdorff. Let $ X_1 $ be a locally convex topological vector space such that it densely embeds in $ X $. Then $ C(\Omega, X_1) $ is dense in $ C(\Omega, X) $ in the uniform topology.
\end{exa}
\begin{proof}
	Here, note that even $ h \in C(\Omega, X) $ such that $ h(\Omega) \subset X_1 $, this does not mean $ h \in C(\Omega, X_1) $. First assume $ \Omega $ is Hausdorff. Then the self-adjoint algebra $ C(\Omega, \mathbb{C}) $ separates $ \Omega $ and so $ C(\Omega, \mathbb{C}) \otimes X_1 \subset C(\Omega, X_1) $ is dense in $ C(\Omega, X) $. Now suppose $ X $ is Hausdorff. The proof given in \autoref{thm:dis} shows $ [x]_{C(\Omega, \mathbb{C})} \subset \rho_{C(\Omega, X)} $, and so by \autoref{thm:nonseperate} (or \autoref{thm:A}) we also have $ C(\Omega, \mathbb{C}) \otimes X_1 \subset C(\Omega, X_1) $ is dense in $ C(\Omega, X) $. The proof is complete.
\end{proof}

Write $ C_{c}(\Omega, X) = \{ f \in C(\Omega, X) : \supp f ~\text{is compact} \} $.

\begin{exa}
	Let $ \Omega $ be locally compact and $ C_{c}(\Omega, \mathbb{C})(t) \neq 0 $ for all $ t \in \Omega $ (e.g. $ \Omega $ is Hausdorff or $ \Omega $ is compact). Further, assume $ X $ is Hausdorff if $ \Omega $ is not Hausdorff. If $ X_0 $ is dense in $ X $, then $ C_{c}(\Omega, \mathbb{C}) \otimes X_0 $ is dense in $ CV_0(\Omega, X) $. In particular, $ C_{c}(\Omega, X) $ is dense in $ CV_0(\Omega, X) $ and $ C_{c}(\Omega, \mathbb{C}) $ is dense in $ CV_0(\Omega, \mathbb{C}) $.
\end{exa}
\begin{proof}
	Let $ A = C_{c}(\Omega, \mathbb{R}) $ and assume $ \Omega $ is not Hausdorff.
	Note that for all $ t \in \Omega $, $ A(t) \neq 0 $. This implies there is $ h \in C_{c}(\Omega, \mathbb{R}) $ such that $ h(x) = h(y) \neq 0 $ where $ y \in [x]_{A} $. Let $ g \in CV_0(\Omega, X) $ and $ x^* \in X^* $ (the dual space of $ X $), then $ f(t) = h(t)\mathrm{Re}x^*(g(t)) \in C_{c}(\Omega, \mathbb{R}) $ and so $ f(x) = f(y) $. This yields that $ \mathrm{Re}x^*g(x) = \mathrm{Re}x^*g(y) $ and then $ g(x) = g(y) $ (by the Hausdorff property of $ X $). Therefore, by \autoref{thm:nonseperate}, $ C_{c}(\Omega, \mathbb{C}) \otimes X_0 $ is dense in $ CV_0(\Omega, X) $ as $ A $ is a subalgebra and $ A (C_{c}(\Omega, \mathbb{C}) \otimes X_0) \subset C_{c}(\Omega, \mathbb{C}) \otimes X_0 $.

	If $ \Omega $ is Hausdorff, then $ [x]_{A} = \{ x \} $. The result follows from \autoref{thm:SW} directly. This completes the proof.
\end{proof}

In the following, we discuss a vector-valued version of Bernstein problem but in a more concrete way. Let $ \gamma: \mathbb{R} \to \mathbb{R}_+ $ be continuous such that $ \lim_{|t| \to \infty} |t|^n \gamma(t) \to 0 $ for all $ n \in \mathbb{N} $. The Bernstein problem asks for what function $ \gamma $, $ \mathbb{C}[z] $ is dense in $ CV^b_0(\mathbb{R}, \mathbb{C}) $ where $ V^b = \{\gamma\} $; if it is so, then such $ \gamma $ is called a \emph{fundamental weight} \cite{Nac65}. This problem was solved independently by Achieser, Mergelyan, and Pollard; see e.g. \cite{Lub07} for details. Now we concentrate under what Nachbin family (weights) $ V $, the polynomial algebra $ \mathbb{C}[z] \otimes X $ is dense in $ CV_0(\mathbb{R}, X) $. The following result is motivated by \cite[Theorem 2]{Nac65}.

\begin{thm}\label{thm:Ber}
	Let $ W $ be a linear subspace of $ CV_0(\mathbb{R}, X) $ such that it is a $ \mathbb{C}[z] $-module (i.e. $ \mathbb{C}[z] W \subset W $).
	Assume for every $ v \in V $, $ p \in \mathcal{A} $ and $ f \in W $, there is a fundamental weight $ \gamma $ such that
	\[
	v(t)p(f(t)) \leq \gamma(t), ~ t \in \mathbb{R}.
	\]
	Then for $ f \in CV_0(\mathbb{R}, X) $, $ f(t) \in \overline{W(t)} $ for all $ t \in \mathbb{R} $ if and only if $ f \in \overline{W} $.
\end{thm}

\begin{proof}
	Here, we use the argument due to \cite{Nac65}. Let $ W_1 = \mathrm{span} \{ b u: b \in C_b(\mathbb{R}, \mathbb{C}), u \in W \} $, then $ C_b(\mathbb{R}, \mathbb{C}) W_1 \subset W_1 $ and $ W \subset W_1 \subset CV_0(\mathbb{R}, X) $. Take $ f \in CV_0(\mathbb{R}, X) $ such that $ f(t) \in \overline{W(t)} $ for all $ t \in \mathbb{R} $. By \autoref{cor:localization} (or \autoref{thm:A}), we have $ f \in \overline{W}_1 $. So for every $ \varepsilon > 0 $, $ v \in V $ and $ p \in \mathcal{A} $, there is $ g \in W_1 $ such that $ v(t)p(f(t) - g(t)) \leq \varepsilon $ for all $ t \in \mathbb{R} $. Write $ g(t) = \sum_{i = 1}^{n} b_i(t) u_i(t) $ where $ b_i \in C_b(\mathbb{R}, \mathbb{C}) $ and $ u_i \in W $. Now by assumption we have fundamental weights $ \gamma_i $ such that
	\[
	v(t)p(u_i(t)) \leq \gamma_i(t), ~ t \in \mathbb{R}.
	\]
	Note that $ C_b(\mathbb{R}, \mathbb{C}) $ is contained in $ CV^b_0(\mathbb{R}, \mathbb{C}) $ for $ V^b = \{\gamma_i\} $; in particular, there is $ p_i \in \mathbb{C}[z] $ such that
	\[
	\gamma_i(t)|b_i(t) - p_i(t)| \leq \varepsilon / n, ~ t \in \mathbb{R}.
	\]
	Let $ f_1(t) = \sum_{i = 1}^{n} p_i(t) u_i(t) $. Then $ f_1 \in W $ (as $ \mathbb{C}[z] W \subset W $) and
	\begin{align*}
	v(t)p(f(t) - f_1(t)) & \leq v(t)p(f(t) - g(t)) + v(t)p(g(t) - f_1(t)) \\
	& \leq \varepsilon + v(t)p(\sum_{i = 1}^{n} \{b_i(t) - p_i(t)\} u_i(t)) \\
	& \leq \varepsilon + \sum_{i = 1}^{n} |b_i(t) - p_i(t)| v(t) p(u_i(t)) \\
	& \leq \varepsilon + \sum_{i = 1}^{n} \gamma_i(t)|b_i(t) - p_i(t)| \leq 2 \varepsilon,
	\end{align*}
	i.e., $ f \in \overline{W} $. The proof is complete.
\end{proof}

\begin{exa}\label{exa:ber}
	Let $ V^{\alpha} = \{ W_{\alpha} \} $ where $ W_{\alpha}(t) = \exp^{- |t|^{\alpha}} $ (the so-called Freud’s weight), and $ X_0 $ a dense linear subspace of $ X $. If $ \alpha > 1 $, then $ \mathbb{C}[z] \otimes X_0 $ is dense in $ CV^{\alpha}_0(\mathbb{R}, X) $.
\end{exa}

\begin{proof}
	It is well known that the Freud’s weight $ W_{\alpha} $ is a fundamental weight if and only if $ \alpha \geq 1 $ (see e.g. \cite[Corollary 1.5]{Lub07}). For $ f(t) = p(t)x $ where $ p \in \mathbb{C}[z] $ and $ x \in X_0 $, since $ \mathbb{C}[z] \subset CV^{1}_0(\mathbb{R}, \mathbb{C}) $, we have $ W_{1}(t) p(f(t)) \leq C_{f} < \infty $ for all $ t \in \mathbb{R} $ where $ C_{f} > 0 $ is some constant depending on $ f $. Now we have for $ \alpha > 1 $
	\[
	W_{\alpha} (t) p(f(t) ) \leq C_1W_1(t)W_1(t)p(f(t)) \leq C_1C_{f}W_1(t),
	\]
	where $ \sup_{t}W_{\alpha} (t) (W_1(t))^{-2} \leq C_1 $ (as $ \alpha > 1 $). Therefore, the above inequality holds for all $ f \in \mathbb{C}[z] \otimes X_0 $. The result now follows from \autoref{thm:Ber}.
\end{proof}

\subsection{approximation with interpolation} \label{sub:interpolation}

\begin{lem}\label{lem:ai}
	Let $ A \subset C(\Omega, [0,1]) $ and $ W $ a topological vector subspace of $ C(\Omega, X) $.
	Take a subset $ W_0 $ of $ CV_0 W $. Assume $ A $ is a multiplier of $ W_0 $ and is symmetric (i.e. $ [t]_A = \{ t \} $ for each $ t \in \Omega $).
	Let $ S, B \subset \Omega $, $ T_0 = \{ t_i: i = 1, 2, \ldots, n \} \subset \Omega $ a finite subset, $ x_i \in W_0(t_i) $, $ i = 1,2,\ldots,n $, and $ X_0 $ a convex subset of $ X $. Set
	\[
	W_{ai} = \{ u \in W_0: u(t_i) = x_i, i = 1, 2, \ldots, n, u|_S = 0, u(B) \subset X_0 \}.
	\]
	Then $ A $ is also a multiplier of $ W_{ai} $. Particularly, $ f(t) \in \overline{W_{ai}(t)} $ for all $ t \in \Omega $ if and only if $ f \in \overline{W_{ai}} $.
\end{lem}

\begin{proof}
	For $ u_1, u_2 \in W_{ai} $ and $ \varphi \in A $, let $ u = \varphi u_1 + (1 - \varphi) u_2 $. Then $ u \in W_0 $ and $ u(t_i) = x_i $, $ i = 1,2,\ldots,n $, $ u|_S = 0 $. Due to that $ \varphi(\Omega) \subset [0,1] $ and $ X_0 $ is convex, we also have $ u(B) \subset X_0 $, i.e., $ u \in W_{ai} $. So $ A $ is a multiplier of $ W_{ai} $. Now by \autoref{cor:localization}, $ f(t) \in \overline{W_{ai}(t)} $ for all $ t \in \Omega $ if and only if $ f \in \overline{W_{ai}} $.
\end{proof}

The above lemma has some striking consequences. We list some of them in the following. When $ B = \emptyset $, this case was already discussed in \cite[Section 4]{Pro94} and \cite{PK02}.
The following results (\autoref{thm:ai0} and \autoref{cor:ai1}) generalize \cite[Theorem 16]{Tim05}.

For $ X_0 \subset X $, let $ co(X_0) $ denote the convex hull of $ X_0 $, i.e., $ co(X_0) $ equals the closure of
\[
\left\{ \sum_{i = 1}^{n} s_i x_i: x_i \in X_0, 0 \leq s_i \leq 1, i = 1,2, \ldots, n, \sum_{i = 1}^{n} s_i = 1, n \in \mathbb{N} \right\}.
\]

\begin{thm}\label{thm:ai0}
	Let $ T_0 $ be a finite subset of $ \Omega $.
	\begin{enumerate}[(a)]
		\item \label{aia} If $ \Omega $ is a locally compact Hausdorff space, then for every $ f \in CV_0(\Omega, X) $, $ \varepsilon > 0 $, $ v \in V $ and $ p \in \mathcal{A} $, there is $ u \in C_{c}(\Omega, \mathbb{C}) \otimes X \subset CV_0(\Omega, X) $ such that $ |f - u|_{v, p}< \varepsilon $, $ u|_{T_0} = f|_{T_0} $, $ u(\Omega) \subset co(f(\Omega) \cup \{0\}) $, and $ \supp u \subset \supp f $.
		\item \label{aib} If $ \Omega $ is a completely regular Hausdorff space, then for every $ f \in C_{b}(\Omega, X) $, $ \varepsilon > 0 $, $ v \in C_0(\Omega, \mathbb{R}_+) $ and $ p \in \mathcal{A} $, there is $ u \in C_{b}(\Omega, \mathbb{C}) \otimes X $ such that $ |f - u|_{v,p} < \varepsilon $, $ u|_{T_0} = f|_{T_0} $, $ u(\Omega) \subset co(f(\Omega)) $, and $ \supp u \subset \supp f $.
	\end{enumerate}
\end{thm}

\begin{proof}
	Let $ S = \overline{ \Omega \setminus \supp f } $. Without loss of generality, $ T_0 \cap S \neq 0 $. Let $ B = \Omega $, $ T_{0} = \{ t_i: i = 1,2,\ldots,n \} $ and $ x_i = f(t_i) $. Take $ W_0 = C_{c}(\Omega, \mathbb{C}) \otimes X $ and $ X_0 = co(f(\Omega) \cup \{0\}) $ for the case (a), and $ W_0 = C_{b}(\Omega, \mathbb{C}) \otimes X $ and $ X_0 = co(f(\Omega)) $ for the case (b). For $ A = C(\Omega, [0,1]) $, it is a multiplier of $ W_0 $ and is symmetric.
	Now let $ W_{ai} $ be defined as in \autoref{lem:ai}. We need to show $ f \in \overline{W_{ai}} $ and then it suffices to show $ f(t) \in W_{ai}(t) $ for all $ t $ (by \autoref{lem:ai}). If $ t \in T_0 \cup S $, then this is true. Let $ t \in \Omega \setminus \{T_0 \cup S\} $ and $ t_0 = t $.
	By the assumption on $ \Omega $, there are $ \varphi_i $ and open $ V_i $, $ i = 0, 1, \ldots, n $, such that $ 0 \leq \varphi_i \leq 1 $, $ 0 \leq \sum_{i = 1}^{n} \varphi_i \leq 1 $, and
	\[
	\{ t_i \} \prec \varphi_i \prec V_i, ~ (\cup_{i = 0}^{n} V_i) \cap S = \emptyset, ~V_{i} \cap V_{j} = \emptyset ~(i \neq j),
	\]
	i.e., $ \varphi_i(t_i) = 1 $ and $ \supp \varphi_i \subset V_i $; in addition, for the case (a), $ \overline{V_i} $ is compact. For a proof, see e.g. \cite[Theorem 2.13]{Rud87} for the case (a); the case (b) can be proved similarly. Define
	\begin{enumerate}[(1)]
		\item $ \tilde{u}(s) = \sum_{i = 0}^n \varphi_i(s) f(t_i) $ for the case (a); and
		\item $ \tilde{u}(s) = \sum_{i = 0}^n \varphi_i(s) f(t_i) + ( 1 - \sum_{i = 0}^n \varphi_i(s) ) f(t') $ for the case (b) where $ t' $ is chosen arbitrarily if $ S = \emptyset $, and $ t' \in S $ if $ S \neq \emptyset $.
	\end{enumerate}
	So $ \tilde{u}(\Omega) \subset X_0 $, $ \tilde{u}|_{S} = 0 $ ($ \Leftrightarrow \supp \tilde{u} \subset \supp f $), $ \tilde{u}|_{T_0} = f|_{T_0} $ and $ \tilde{u} \in W_0 $, i.e., $ \tilde{u} \in W_{ai} $ and $ f(t) = \tilde{u}(t) \in W_{ai}(t) $. The proof is complete.
\end{proof}

For two Nachbin families $ V, V' $, we write $ V \leq V' $ if for every $ v \in V $ there is $ v' \in V' $ such that $ v \leq v' $.
\begin{cor}\label{cor:ai1}
	Suppose $ X $ is Hausdorff (but $ \Omega $ is any topological space).
	\begin{enumerate}[(a)]
		\item (See Timofte \cite[Theorem 16]{Tim05}) Let two Nachbin families $ V, V' $ satisfy $ V' \subset C(\Omega, \mathbb{R}_+) $ and $ V \leq V' $. The conclusion \eqref{aia} in \autoref{thm:ai0} also holds if $ f \in C_0(\Omega, X) \cap CV'_0(\Omega, X) $ (in the topology of $ CV_0(\Omega, X) $); particularly, if $ f \in C_0(\Omega, X) $ in the uniform topology.

		\item The conclusion \eqref{aib} in \autoref{thm:ai0} also holds (but without assuming $ \Omega $ is completely regular and Hausdorff).
	\end{enumerate}
\end{cor}

\begin{proof}
	We will use the following fact due to Stone and \^{C}ech (see e.g. \cite[Theorem 3.9]{GJ60}).
	\begin{enumerate}[$ \bullet $]
		\item For any topological space $ \mathcal{T} $, there exist a completely regular Hausdorff space $ \mathfrak{T} $ and a continuous map $ \pi $ of $ \mathcal{T} $ onto $ \mathfrak{T} $, such that the map $ g \mapsto g \circ \pi $ is an isomorphism of $ C(\mathfrak{T}, \mathbb{R}) $ onto $ C(\mathcal{T}, \mathbb{R}) $.
	\end{enumerate}
	In fact, the space is given by $ \mathfrak{T} = \{ [x]_{C(\mathcal{T}, \mathbb{R})}: x \in \mathcal{T} \} $ and the map $ \pi $ is constructed by $ x \mapsto [x]_{C(\mathcal{T}, \mathbb{R})} $. The topology of $ \mathfrak{T} $ is the weakest one such that all $ h: \mathfrak{T} \to \mathbb{R} $ satisfying $ h \circ \pi \in C(\mathcal{T}, \mathbb{R}) $ are continuous. So for $ h \in C(\mathfrak{T}, \mathbb{R}) $ and $ \varepsilon > 0 $, the set
	\[
	\{ y: |h(y) - h(y_0)| < \varepsilon \}
	\]
	is a neighborhood of $ y_0 \in \mathfrak{T} $. In the following, we will write $ \mathfrak{T} = \widehat{\mathcal{T}} $ if $ \mathfrak{T} $ is constructed through this way.

	(a)
	Note that since $ f \in C_0(\Omega, X) $ and $ X $ is Hausdorff, for every open neighborhood $ \mathcal{N} $ of $ 0 $, the closure of $ f^{-1}(X \setminus \mathcal{N}) $ is compact. Particularly, for $ \Omega_0 \triangleq f^{-1}(X \setminus \{0\}) $ and its corresponding space $ \widehat{\Omega}_0 $, they are locally compact. We show this is true for $ \widehat{\Omega}_0 $ as follows. Let $ \widehat{f} \in C(\widehat{\Omega}_0, X) $ such that $ f = \widehat{f} \circ \pi $. Take $ [t_0] \in \widehat{\Omega}_0 $ and $ p_0 \in \mathcal{A} $ such that $ p_0(\widehat{f}([t_0])) \neq 0 $. Then $ \widehat{V}_0 \triangleq \{ [t]: |p_0(\widehat{f}([t_0])) - p_0(\widehat{f}([t])) | \leq \varepsilon_0 \} $ is a (closed) neighborhood of $ [t_0] $ as $ p_0\circ f = p_0 \circ \widehat{f} \circ \pi \in C(\Omega_0, \mathbb{R}) $ where $ \varepsilon_0 = p_0(\widehat{f}([t_0])) / 2 $. Note that $ \pi^{-1} \widehat{V}_0 = \{ t: |p_0(f(t_0)) - p_0(f(t)) | \leq \varepsilon_0 \} $ which is compact as $ f \in C_0(\Omega, X) $. It yields that $ \widehat{V}_0 $ is compact. The above argument also shows that $ \widehat{f} \in C_0(\widehat{\Omega}_0, X) $ and $ \pi: \Omega_0 \to \widehat{\Omega}_0 $ is proper.

	Set $ \widehat{V}^+ = \{ \widehat{v}' \in C(\widehat{\Omega}_0, \mathbb{R}_+): \widehat{v}' \circ \pi \in V'|_{\Omega_0} \} $. Note that $ C_c(\widehat{\Omega}_0, X) \subset C\widehat{V}^+_0(\widehat{\Omega}_0, X) $ and $ CV'_0(\Omega_0, X) \subset CV_0(\Omega_0, X) $.
	Without loss of generality, suppose $ T_0 \subset \supp f $. Let $ \widehat{T}_0 = \pi(T_0) $, and $ \varepsilon > 0 $, $ v \in V $, $ p \in \mathcal{A} $. Choose $ v' \in V' $ and $ \widehat{v}' \in \widehat{V}^+ $ such that $ v \leq v' $ and $ \widehat{v}' \circ \pi = v' $.
	Since $ \widehat{f} \in C\widehat{V}^+_0(\widehat{\Omega}_0, X) $ (due to $ f \in CV'_0(\Omega, X) $ and $ \pi $ being proper), by \autoref{thm:ai0} \eqref{aia}, there is $ \widehat{u} \in C_{c}(\widehat{\Omega}_0, \mathbb{C}) \otimes X \subset C\widehat{V}^+_0(\widehat{\Omega}_0, X) $ such that $ |\widehat{f} - \widehat{u}|_{\widehat{v}', p}< \varepsilon $, $ \widehat{u}|_{\widehat{T}_0} = \widehat{f}|_{\widehat{T}_0} $, $ \widehat{u}(\widehat{\Omega}_0) \subset co(\widehat{f}(\widehat{\Omega}_0) \cup \{0\}) $, and $ \supp \widehat{u} \subset \supp \widehat{f} $.

	Define $ u = \widehat{u} \circ \pi \in C(\Omega_0, \mathbb{C}) \otimes X $. Then $ u|_{T_0} = f|_{T_0} $, $ u(\Omega_0) \subset co(f(\Omega_0) \cup \{0\}) $, and $ \supp u \subset \supp f $. Moreover, one can easily see that $ \supp u $ is compact (as $ \supp \widehat{u} $ is compact and $ \pi $ is proper), i.e., $ u \in C_c(\Omega_0, \mathbb{C}) \otimes X $; in addition, since $ \Omega_0 $ is open in $ \Omega $, we can assume $ u \in C_c(\Omega, \mathbb{C}) \otimes X $.
	Now we have
	\[
	|f - u|_{v, p} \leq |f - u|_{v', p} = |f - u|_{v', p, \Omega_0} = |\widehat{f} - \widehat{u}|_{\widehat{v}', p} < \varepsilon.
	\]

	(b) Note that for $ y \in [x]_{C(\Omega, \mathbb{R})} $, we have $ h(x) = h(y) $ where $ h \in C(\Omega, X) $ as $ X $ is Hausdorff (see also the proof of \autoref{thm:dis}).
	So for $ f \in C_b(\Omega, X) $, we have a unique $ \widehat{f} \in C_{b}(\widehat{\Omega}, X) $ such that $ f = \widehat{f} \circ \pi $; similarly for $ v \in C_0(\Omega, \mathbb{R}_+) $, there is $ \widehat{v} \in C_0(\widehat{\Omega}, \mathbb{R}_+) $ such that $ \widehat{v} \circ \pi = v $ (see the proof in (a)). As $ \widehat{\Omega} $ is a completely regular Hausdorff space, by \autoref{thm:ai0} \eqref{aib}, we also have $ \widehat{u} \in C_{b}(\widehat{\Omega}, \mathbb{C}) \otimes X $ such that $ |\widehat{f} - \widehat{u}|_{\widehat{v}, p} < \varepsilon $, $ \widehat{u}|_{\pi(T_0)} = \widehat{f}|_{\pi(T_0)} $, $ \widehat{u}(\widehat{\Omega}) \subset co(\widehat{f}(\widehat{\Omega})) $, and $ \supp \widehat{u} \subset \supp \widehat{f} $. Let $ u = \widehat{u} \circ \pi $ which is the desired function. The proof is complete.
\end{proof}

In the following, we focus on a concrete problem, i.e., to determine the closure of an operator's domain, which often occurs in the operator semigroup theory, e.g., in the study of delay equations by using operator semigroup theory (see e.g. \cite{EA06, EN00} and the references therein); the following is such an example.

\begin{exa}\label{exa:delay0}
	Let $ A : D(A) \subset X \to X $ be a linear (unbounded) operator where $ D(A) $ is its domain, $ X_0 = \overline{D(A)} $ and
	\[
	\mathcal{C}_{A} = \{ \varphi \in C([-1, 0], X): \varphi(0) \in X_0 \}.
	\]
	Consider
	\begin{gather*}
	D_{1,0} = \{ \varphi \in C^{1}([-1, 0], X): \varphi(0) \in D(A), \varphi'(0) = A\varphi(0) \}.
	\end{gather*}
	We will show $ \overline{D_{1,0}} = \mathcal{C}_{A} $ below.
\end{exa}

The following result about polynomial seems well known.
\begin{lem}\label{lem:inter}
	Given $a_{ij}, x_i \in \mathbb{C}$, there is a polynomial $p(x)$ satisfying $p^{(j)}(x_i) = a_{ij}$, where $x_i$ are different from each other, $i = 1,2,\ldots,n,~j=0,1,\ldots,m$.
\end{lem}
\begin{proof}
	A quick proof can be as follows. For $ m = 0 $, this is well known. Suppose this lemma holds for $ m = s - 1 $, i.e., there is a polynomial $p$ such that $p^{(j)}(x_i) = a_{ij},~ j=0,1,\ldots,s-1$. Let us consider the case $ m = s $.
	Take the polynomial
	\[
	p_1(x) = \prod\limits_{i=1}^{n}(x - x_i)^s;
	\]
	note that $p_1^{(s)}(x_i) \neq 0, p^{(k)}_1(x_i) = 0$, $k = 0,1,\ldots,s-1$, $ i = 1,2,\ldots n $. Set
	\[
	f(x) = p(x) + b(x)p_1(x),
	\]
	where $b(x)$ is a polynomial which will be chosen later. One gets
	\[
	f^{(j)}(x_i) = p^{(j)}(x_i) + (bp_1)^{(j)}(x_i) = p^{(j)}(x_i) = a_{ij},~ j = 0,1,\ldots,s-1,
	\]
	and
	\[
	f^{(s)}(x_i) = p^{(s)}(x_i) + b(x_i)p_1^{(s)}(x_i).
	\]
	Thus, it suffices to take $b(x)$ such that $b(x_i) = \frac{a_{is} - p^{(s)}(x_i)}{p_1^{(s)}(x_i)}$, giving the desired polynomial such that $f^{(s)}(x_i) = a_{is}$. The proof is complete.
\end{proof}

\begin{rmk}
	The same argument in fact yields the following result.
	\begin{enumerate}[$ \bullet $]
		\item Given different $ n $ points $ x_i $, $ i = 1,2,\ldots,n $, and $ m_{ij_i}(\lambda) \in \mathbb{C} $, then there is a unique polynomial $ P_{\lambda} $ with degree less that $ j_1 + j_2 + \cdots + j_n + n - 1 $ such that
		\[
		P_{\lambda}(x_i) = m_{i0}(\lambda), P'_{\lambda}(x_i) = m_{i1}(\lambda), \cdots, P^{(j_i)}_{\lambda}(x_i) = m_{ij_i}(\lambda), ~ i = 1,2,\ldots,n.
		\]
		Moreover, if all $ \lambda \mapsto m_{ij_i}(\lambda) $ are $ C^{k} $ and $ P_{\lambda}(t) = \sum a_{j}(\lambda) t^{j} $, then all $ \lambda \mapsto a_{j}(\lambda) $ are $ C^{k} $.
	\end{enumerate}
\end{rmk}

\begin{lem}\label{lem:closure}
	Let $ \Omega $ be a completely regular Hausdorff space, $ T_i $ a finite subset of $ \Omega $, $S_{j} \subset \Omega $ such that $ \overline{S_{j}} $ is compact, $ X_j $ a convex subset of $ X $, $ B_i \subset X $, $ j = 1,2, \ldots,n $, $ i = 1,2,\ldots,m $; in addition $ \overline{ S_{i_1} } \cap \overline{ S_{i_2} } = \emptyset $ ($ i_1 \neq i_2 $). Set
	\[
	\mathcal{K}_{T_{i},B_{i}; S_{j}, X_j} \triangleq \{ g \in C_{b}(\Omega, E): g(T_{i}) \subset B_i, ~1 \leq i\leq m,~ g(S_{j}) \subset X_{j}, ~1\leq j \leq n \}.
	\]
	Then $\overline{\mathcal{K}_{T_{i},B_{i}; S_{j}, X_j}} = \mathcal{K}_{{T_{i}},\overline{B_{i}}; \overline{S_{j}}, \overline{X_j}}$ (in the uniform topology).
\end{lem}

\begin{proof}
	Following the notation suggested in \cite{Rud87}, for $ h \in C(\Omega, [0,1]) $, let $ M \prec h \prec N $ denote $ h|_{M} = 1 $ and $ \supp h \subset N $, where $ M \subset N \subset \Omega $.

	For simplicity, we only consider $ n = m = 1 $. Without loss of generality, we can assume $ T_1 \cap \overline{S_1} = \emptyset $. Write $ T_1 = \{s_i: i = 1,2,\ldots,k\} $. Clearly, $ \overline{\mathcal{K}_{T_{1},B_{1}; S_{1}, X_1}} \subset \mathcal{K}_{{T_{1}},\overline{B_{1}}; \overline{S_{1}}, \overline{X_1}} $. Let us show $ \mathcal{K}_{{T_{1}},\overline{B_{1}}; \overline{S_{1}}, \overline{X_1}} \subset \overline{\mathcal{K}_{T_{1},B_{1}; S_{1}, X_1}} $. Take $ g \in \mathcal{K}_{{T_{1}},\overline{B_{1}}; \overline{S_{1}}, \overline{X_1}} $. Let $ \varepsilon > 0 $, $ p \in \mathcal{A} $, and define
	\[
	V_s(\varepsilon) \triangleq \{ y \in \Omega:  p(g(s) - g(y))  < \varepsilon \}, ~ s \in \Omega.
	\]
	Since $ \overline{ S_1 } $ is compact, there are $\tilde{s}_i \in \overline{S_1}$ ($k+1 \leq i \leq k+k_1$) such that $ \overline{ S_1 } \subset \bigcup_{i=k+1}^{k+k_1}V_{\tilde{s}_i} (\varepsilon / 2) $. Take $ s_i \in V_{\tilde{s}_i} (\varepsilon) \cap S_1 \neq \emptyset $ (as $\tilde{s}_i \in \overline{S_1}$). Then
	\[
	p(g(s_i) - g(y)) \leq 2\varepsilon, ~\forall y \in V_{\tilde{s}_i} (\varepsilon).
	\]
	Choose open sets $ \widetilde{V}_{i} $ such that $ s_i \in \widetilde{V}_{i} \subset V_{s_i} (\varepsilon) $ ($ 1 \leq i \leq k $) $ (\bigcup_{i = 1}^{k} \widetilde{V}_{i}) \cap \overline{S_1} = \emptyset $ and $ \widetilde{V}_{i} \cap \widetilde{V}_{j} = \emptyset $ ($ i \neq j $). As $ g(s_i) \in \overline{X_1} $, $ i \geq k+1 $, there are $ b_i \in X_{1} $ such that $ p(g(s_i) - b_i) < \varepsilon $; similarly, for $ 1 \leq i \leq k $, there are $ b_i \in B_{1} $ such that $ p(g(s_i) - b_i) < \varepsilon $. Let $ K_i = \overline{ S_1 } \cap \overline{V_{\tilde{s}_i} (\varepsilon / 2)} $, $ i \geq k+1 $. Then $ K_i $ is compact and $ K_i \subset V_{\tilde{s}_i} (\varepsilon) $. Since $ \Omega $ is completely regular and Hausdorff, we have $ h_i \in C(\Omega, [0,1]) $ such that
	\[
	\{ s_i \} \prec h_i \prec \widetilde{V}_i ~ ( 1 \leq i \leq k ), ~K_i \prec h_i \prec V_{\tilde{s}_i} (\varepsilon)~ (k+1 \leq i \leq k+k_1).
	\]
	Write
	\[
	\varphi_1 = h_1, ~\varphi_2 = (1 - h_1)h_2,~ \cdots, ~\varphi_i = (1 - h_1)(1 - h_2)\cdots(1 - h_{i - 1})h_i, ~ 1 \leq i \leq k+k_1.
	\]
	Then $ \varphi_{i}(s_i) = 1 $, $ \supp \varphi_{i} \subset \widetilde{V}_{i} $ ($ 1 \leq i \leq k $), and $ \sum_{i = k+1}^{k+k_1} \varphi_{i}(t) = 1 $ if $ t \in S_1 $; in addition, $ 0 \leq \sum_{i = 1}^{k+k_1} \varphi_{i} \leq 1 $. That is, $ \{ \varphi_{i} \} $ is a unit partition of $ T_1 \cup S_1 $. Define
	\[
	g_{\varepsilon, p}(t) = \sum_{i = 1}^{k+k_1} \varphi_{i}(t) b_i + (1 - \sum_{i = 1}^{k+k_1} \varphi_{i}(t)) g(t).
	\]
	We have $ g_{\varepsilon, p}(s_i) = b_i \in B_1 $ ($ 1 \leq i \leq k $), and for $ t \in S_1 $, $ g_{\varepsilon, p}(t) = \sum_{i = k+1}^{k+k_1} \varphi_{i}(t) b_i \in X_1 $ (as $ X_1 $ is convex), that is, $ g_{\varepsilon, p} \in \mathcal{K}_{T_{1},B_{1}; S_{1}, X_1} $. Moreover,
	\begin{align*}
	p(g_{\varepsilon, p}(t) - g(t)) & \leq \sum_{i=1}^{k+k_1}\varphi_i(t) p (b_i - g(t)) \\
	& =  \sum_{i=1}^{k+k_1} \varphi_i(t) p(b_i - g({s_i}) + g({s_i}) - g(t)) \\
	& \leq  3\varepsilon.
	\end{align*}
	Therefore $ g \in \overline{\mathcal{K}_{T_{1},B_{1}; S_{1}, X_1}} $, and this completes the proof.
\end{proof}

\begin{lem}\label{lem:appP}
	Let $ - \infty < a \leq t_i \leq b < \infty $, $ B_{ij} \subset X $, $ i = 0, 1,2,\ldots,n $, $ j = 0, 1,2,\ldots,m $, and
	\begin{gather*}
	\mathcal{P}_{0, X} = \{ f|_{[a, b]}: f \in \mathbb{C}[z] \otimes X, f^{(j)}(t_i) \in B_{ij}, 0 \leq i \leq n, ~0 \leq j \leq m \},\\
	\mathcal{P}_{1, X} = \{ f|_{[a, b]}: f \in \mathbb{C}[z] \otimes X, f^{(j)}(t_i) \in B_{ij}, 0 \leq i \leq n, ~1 \leq j \leq m \},\\
	\mathcal{K}_{0, X} = \{ f \in C([a, b], X): f(t_i) \in \overline{ B_{ij} }, 0 \leq i \leq n \}.
	\end{gather*}
	Then $ \overline{\mathcal{P}_{0,X}} = \mathcal{K}_{0, X} $ and $ \overline{\mathcal{P}_{1, X}} = C([a, b], X) $ (in the uniform topology).
\end{lem}
\begin{proof}
	Let $ A_0 = \{ p|_{[a, b]}: p \in \mathbb{C}[z], p^{(j)}(t_i) = 0, 0 \leq i \leq n, ~0 \leq j \leq m \} $ and $ S = \{t_i:1 \leq i \leq n \} $. First assume $ B_{ij} = \{0\} $. In this case, we have $ A_0 \mathcal{P}_{0, X} \subset \mathcal{P}_{0, X} $ and all the assumptions for $ S $, $ A = A_0 $ and $ G = \mathcal{P}_{0, X} $ in \autoref{thm:SW} are satisfied; conditions (iii) (iv) hold due to \autoref{lem:inter}. This shows that $ \overline{\mathcal{P}_{0,X}} = \mathcal{K}_{0, X} $.

	Consider the general case; without loss of generality, let $ B_{ij} \neq \emptyset $. For brevity, let $ n = 1, m = 1 $. By \autoref{lem:closure}, it suffices to consider the approximation of $ g \in C([a, b], X) $, $ g(t_i) = b_{i0} $ where $ b_{ij} \in B_{ij} $. Using \autoref{lem:inter}, we can take four polynomials $ p_{ij} $ such that $ p_{00}(t_0) = 1 $, $ p_{10}(t_1) = 1 $, $ p'_{01}(t_0) = 1 $, $ p'_{11}(t_1) = 1 $ and others of $ p^{(l)}_{ij}(t_k) $ equal $ 0 $ ($ i,j,l,k \in \{0,1\} $). Set $ f(t) = \sum p_{ij}(t)b_{ij} $. Then $ f^{(j)}(t_i) = b_{ij} \in B_{ij} $. For $ g_0 = g - f $ and every $ p \in \mathcal{A} $, by the above proof, we have polynomials $ g_{n,p} $ such that $ g_{n,p} \rightrightarrows g_0 $ and $ g^{(j)}_{n,p}(t_{i}) = 0 $. Now $ g_{n,p} + f \rightrightarrows g $ and $ g_{n,p} + f \in \mathcal{P}_{0, X} $. That is, $ \overline{\mathcal{P}_{0,X}} = \mathcal{K}_{0, X} $.

	The proof of $ \overline{\mathcal{P}_{1, X}} = C([a, b], X) $ is similar (and easier). This completes the proof.
\end{proof}

Here is a simple application of \autoref{lem:appP} which is a well known result: For every $ \varepsilon > 0 $, $ n, m \in \mathbb{N} $ and every $ f \in C^{m}([0,1], \mathbb{C}) $, $ t_i \in [0,1] $, $ i = 1,2, \ldots,n $, there is $ g \in \mathbb{C}[z] $ such that $ g^{(j)}(t_i) = f^{(j)}(t_i) $ ($ 0 \leq i \leq n $, $ 0 \leq j \leq m $) and $ |f - g| \leq \varepsilon $.

\begin{proof}[Proof of \autoref{exa:delay0}]
	By \autoref{lem:closure}, it suffices to consider the approximation of $ \varphi \in C([a, b], X) $, where $ \varphi(0) \in D(A) $. For every $ p \in \mathcal{A} $, by \autoref{lem:closure}, we know there are $ \varphi_{n,p} \in \mathbb{C}[z] \otimes X $ such that $ \varphi_{n,p} \rightrightarrows \varphi $, $ \varphi_{n,p}(0) = \varphi(0) $ and $ \varphi'_{n,p}(0) = A \varphi(0) $; particularly $ \varphi'_{n,p}(0) = A \varphi_{n,p}(0) $, i.e., $ \varphi_{n,p} \in D_{1,0} $. The proof is complete.
\end{proof}

For $ \lambda \in \mathbb{R} $, write $ \lambda - A = \lambda \cdot \id - A $.

\begin{exa}\label{exa:delay}
	Let $ A, X_0, \mathcal{C}_{A} $ be as in \autoref{exa:delay0}. Assume $ X $ is a Banach space, $ L: C([-1, 0], X) \to X $ is a bounded linear operator and $ \lim_{\lambda \to \infty}\|(\lambda - A)^{-1}\| = 0 $.
	Consider
	\begin{gather*}
	D_1 = \{ \varphi \in C^{1}([-1, 0], X): \varphi(0) \in D(A), \varphi'(0) = A\varphi(0) + L \varphi \}.
	\end{gather*}
	Then $ \overline{D_1} = \mathcal{C}_{A} $.
\end{exa}

\begin{proof}

	By \autoref{lem:closure}, it is sufficient to consider the approximation of $ \varphi \in C([a, b], X) $ such that $ \varphi(0) \in D(A) $. First, by \autoref{exa:delay0}, there are $ \varphi_{n} \in \mathbb{C}[z] \otimes X $ such that $ \varphi_{n} \rightrightarrows \varphi $ and $ \varphi'_{n}(0) = A\varphi_{n}(0) $, $ \varphi_{n}(0) \in D(A) $. Write
	\[
	\epsilon_{\lambda} (t) \triangleq e^{\lambda t},~ (\epsilon_{\lambda} \otimes x) (t) \triangleq \epsilon_{\lambda} (t) x, ~t \in [-1,0], ~\text{and}~ L_{\lambda} x \triangleq L \epsilon_{\lambda} \otimes x, ~x \in X.
	\]
	Set $ \widetilde{\varphi}_{n} = \varphi_{n} - \epsilon_{\lambda_{n}}\otimes x_n $ for some large $ \lambda_{n} > 0 $ and $ x_{n} \in X_0 $ which will be chosen later. Then we have
	\[
	\widetilde{\varphi}'_{n}(0) = \varphi'_{n}(0) - \lambda_{n} x_n = A\varphi_{n}(0) - \lambda_{n} x_n = A(\widetilde{\varphi}_{n}(0) + x_n) - \lambda_n x_n,
	\]
	and so if we let $ (A + L_{\lambda_n} - \lambda_{n}) x_n = L\varphi_{n} $, then $ \widetilde{\varphi}'_{n}(0) = A\widetilde{\varphi}_{n}(0) + L\widetilde{\varphi}_{n} $, i.e., $ \widetilde{\varphi}_{n} \in D_1 $. Note that since $ \|L_{\lambda}\| \leq \|L\| $ for all $ \lambda > 0 $ and $ \|(\lambda - A)^{-1}\| \to 0 $ as $ \lambda \to \infty $, we have $ \lambda - A - L_{\lambda} $ is invertible for large $ \lambda > 0 $ and
	\[
	(\lambda - A - L_{\lambda})^{-1} = (\id - (\lambda - A)^{-1}L_\lambda)^{-1} (\lambda - A)^{-1};
	\]
	in addition, $ \|(\lambda - A - L_{\lambda})^{-1}\| \to 0 $ as $ \lambda \to \infty $. Let $ x_n = (A + L_{\lambda_n} - \lambda_{n})^{-1} L\varphi_{n} $ where $ \lambda_{n} = n $ for large $ n \in \mathbb{N} $. As $ \varphi_{n} \rightrightarrows \varphi $, we get $ x_n \to 0 $ and so $ \sup_{t\in [-1,0]}|(\epsilon_{\lambda_{n}}\otimes x_n)(t)| \leq |x_n| \to 0 $, i.e., $ \widetilde{\varphi}_{n} \rightrightarrows \varphi $. This shows that $ \overline{D_1} = \mathcal{C}_{A} $ and completes the proof.
\end{proof}

\autoref{exa:delay} is related with the following \emph{delay} operator $ A_{d} $ when $ A $ is a \emph{Hille--Yosida operator} (see e.g. \cite[Definition 3.22]{EN00}),
\[
A_{d} \varphi = \dot{\varphi}, ~ \varphi \in D_1 \triangleq D(A_{d}).
\]
In this case, $ \lim_{\lambda \to \infty}\|(\lambda - A)^{-1}\| = 0 $ is satisfied by the definition of Hille--Yosida operator.
For more details, see e.g. \cite{EA06, EN00} and the references therein.

\section*{Acknowledgments}
The author is very indebted to the referees for their meaningful comments which improved the original manuscript.



\begin{bibdiv}
\begin{biblist}

\bib{BD81}{article}{
      author={Brosowski, Bruno},
      author={Deutsch, Frank},
       title={An elementary proof of the {S}tone-{W}eierstrass theorem},
        date={1981},
        ISSN={0002-9939},
     journal={Proc. Amer. Math. Soc.},
      volume={81},
      number={1},
       pages={89\ndash 92},
         url={https://doi.org/10.2307/2043993},
      review={\MR{589143}},
}

\bib{Bis61}{article}{
      author={Bishop, Errett},
       title={A generalization of the {S}tone-{W}eierstrass theorem},
        date={1961},
        ISSN={0030-8730},
     journal={Pacific J. Math.},
      volume={11},
       pages={777\ndash 783},
         url={http://projecteuclid.org/euclid.pjm/1103037116},
      review={\MR{0133676}},
}

\bib{Bor76}{article}{
      author={Borell, Christer},
       title={Approximation on locally convex spaces},
        date={1976},
        ISSN={0020-9910},
     journal={Invent. Math.},
      volume={34},
      number={3},
       pages={215\ndash 229},
         url={https://doi.org/10.1007/BF01403068},
      review={\MR{0413236}},
}

\bib{Buc58}{article}{
      author={Buck, R.~Creighton},
       title={Bounded continuous functions on a locally compact space},
        date={1958},
        ISSN={0026-2285},
     journal={Michigan Math. J.},
      volume={5},
       pages={95\ndash 104},
         url={http://projecteuclid.org/euclid.mmj/1028998054},
      review={\MR{0105611}},
}

\bib{Bur05}{article}{
      author={Burlando, Laura},
       title={The distance to the functions with range in a given set in
  {B}anach spaces of vector-valued continuous functions},
        date={2005},
        ISSN={0021-9045},
     journal={J. Approx. Theory},
      volume={137},
      number={1},
       pages={22\ndash 41},
         url={https://doi.org/10.1016/j.jat.2005.07.003},
      review={\MR{2179621}},
}

\bib{EA06}{incollection}{
      author={Ezzinbi, K.},
      author={Adimy, M.},
       title={The basic theory of abstract semilinear functional differential
  equations with nondense domain},
        date={2006},
   booktitle={Delay differential equations and applications},
      series={NATO Sci. Ser. II Math. Phys. Chem.},
      volume={205},
   publisher={Springer, Dordrecht},
       pages={347\ndash 407},
         url={https://doi.org/10.1007/1-4020-3647-7_9},
      review={\MR{2337821}},
}

\bib{EN00}{book}{
      author={Engel, Klaus-Jochen},
      author={Nagel, Rainer},
       title={One-parameter semigroups for linear evolution equations},
      series={Graduate Texts in Mathematics},
   publisher={Springer-Verlag, New York},
        date={2000},
      volume={194},
        ISBN={0-387-98463-1},
        note={With contributions by S. Brendle, M. Campiti, T. Hahn, G.
  Metafune, G. Nickel, D. Pallara, C. Perazzoli, A. Rhandi, S. Romanelli and R.
  Schnaubelt},
      review={\MR{1721989}},
}

\bib{FdLP84}{article}{
      author={Feyel, D.},
      author={de~La~Pradelle, A.},
       title={Sur certaines extensions du th\'{e}or\`eme d'approximation de
  {B}ernstein},
        date={1984},
        ISSN={0030-8730},
     journal={Pacific J. Math.},
      volume={115},
      number={1},
       pages={81\ndash 89},
         url={http://projecteuclid.org/euclid.pjm/1102708413},
      review={\MR{762203}},
}

\bib{GJ60}{book}{
      author={Gillman, Leonard},
      author={Jerison, Meyer},
       title={Rings of continuous functions},
      series={The University Series in Higher Mathematics},
   publisher={D. Van Nostrand Co., Inc., Princeton, N.J.-Toronto-London-New
  York},
        date={1960},
      review={\MR{0116199}},
}

\bib{Gli63}{article}{
      author={Glicksberg, I.},
       title={Bishop's generalized {S}tone-{W}eierstrass theorem for the strict
  topology},
        date={1963},
        ISSN={0002-9939},
     journal={Proc. Amer. Math. Soc.},
      volume={14},
       pages={329\ndash 333},
         url={https://doi.org/10.2307/2034636},
      review={\MR{0146645}},
}

\bib{Jew63}{article}{
      author={Jewett, R.~I.},
       title={A variation on the {S}tone-{W}eierstrass theorem},
        date={1963},
        ISSN={0002-9939},
     journal={Proc. Amer. Math. Soc.},
      volume={14},
       pages={690\ndash 693},
         url={https://doi.org/10.2307/2034972},
      review={\MR{0152882}},
}

\bib{Kas14}{article}{
      author={Kashimoto, M.~S.},
       title={A note on a {S}tone-{W}eierstrass type theorem for set-valued
  mappings},
        date={2014},
        ISSN={0021-9045},
     journal={J. Approx. Theory},
      volume={182},
       pages={59\ndash 67},
         url={https://doi.org/10.1016/j.jat.2014.03.008},
      review={\MR{3195378}},
}

\bib{Lub07}{article}{
      author={Lubinsky, D.~S.},
       title={A survey of weighted polynomial approximation with exponential
  weights},
        date={2007},
        ISSN={1555-578X},
     journal={Surv. Approx. Theory},
      volume={3},
       pages={1\ndash 105},
      review={\MR{2276420}},
}

\bib{Mac77}{article}{
      author={Machado, Silvio},
       title={On {B}ishop's generalization of the {W}eierstrass-{S}tone
  theorem},
        date={1977},
     journal={Indag. Math.},
      volume={39},
      number={3},
       pages={218\ndash 224},
      review={\MR{0448046}},
}

\bib{Nac65}{article}{
      author={Nachbin, Leopoldo},
       title={Weighted approximation for algebras and modules of continuous
  functions: {R}eal and self-adjoint complex cases},
        date={1965},
        ISSN={0003-486X},
     journal={Ann. of Math. (2)},
      volume={81},
       pages={289\ndash 302},
         url={https://doi.org/10.2307/1970617},
      review={\MR{0176353}},
}

\bib{PB17}{article}{
      author={Paltineanu, Gavriil},
      author={Bucur, Ileana},
       title={Some density theorems in the set of continuous functions with
  values in the unit interval},
        date={2017},
        ISSN={1660-5446},
     journal={Mediterr. J. Math.},
      volume={14},
      number={2},
       pages={Art. 44, 12},
         url={https://doi.org/10.1007/s00009-017-0870-5},
      review={\MR{3619405}},
}

\bib{PK02}{article}{
      author={Prolla, J.~B.},
      author={Kashimoto, M.~S.},
       title={Simultaneous approximation and interpolation in weighted spaces},
        date={2002},
        ISSN={0009-725X},
     journal={Rend. Circ. Mat. Palermo (2)},
      volume={51},
      number={3},
       pages={485\ndash 494},
         url={https://doi.org/10.1007/BF02871856},
      review={\MR{1947469}},
}

\bib{PM73}{article}{
      author={Prolla, Jo\~{a}o~Bosco},
      author={Machado, Silvio},
       title={Weighted {G}rothendieck subspaces},
        date={1973},
        ISSN={0002-9947},
     journal={Trans. Amer. Math. Soc.},
      volume={186},
       pages={247\ndash 258 (1974)},
         url={https://doi.org/10.2307/1996564},
      review={\MR{0402477}},
}

\bib{PM82}{article}{
      author={Prolla, Jo\~{a}o},
      author={Machado, S.},
       title={Weierstrass-{S}tone theorems for set-valued mappings},
        date={1982},
        ISSN={0021-9045},
     journal={J. Approx. Theory},
      volume={36},
      number={1},
       pages={1\ndash 15},
         url={https://doi.org/10.1016/0021-9045(82)90066-1},
      review={\MR{673852}},
}

\bib{Pro71}{article}{
      author={Prolla, Jo\~{a}o~B.},
       title={Bishop's generalized {S}tone-{W}eierstrass theorem for weighted
  spaces},
        date={1971},
        ISSN={0025-5831},
     journal={Math. Ann.},
      volume={191},
       pages={283\ndash 289},
         url={https://doi.org/10.1007/BF01350331},
      review={\MR{0290015}},
}

\bib{Pro77}{book}{
      author={Prolla, Jo\~{a}o~Bosco},
       title={Approximation of vector valued functions},
   publisher={North-Holland Publishing Co., Amsterdam-New York-Oxford},
        date={1977},
        ISBN={0-444-85030-9},
        note={North-Holland Mathematics Studies, Vol. 25, Notas de
  Matem\'{a}tica, No. 61. [Notes on Mathematics, No. 61]},
      review={\MR{0500122}},
}

\bib{Pro88}{article}{
      author={Prolla, Jo\~{a}o~B.},
       title={A generalized {B}ernstein approximation theorem},
        date={1988},
        ISSN={0305-0041},
     journal={Math. Proc. Cambridge Philos. Soc.},
      volume={104},
      number={2},
       pages={317\ndash 330},
         url={https://doi.org/10.1017/S030500410006549X},
      review={\MR{948917}},
}

\bib{Pro94}{article}{
      author={Prolla, Jo\~{a}o~B.},
       title={On the {W}eierstrass-{S}tone theorem},
        date={1994},
        ISSN={0021-9045},
     journal={J. Approx. Theory},
      volume={78},
      number={3},
       pages={299\ndash 313},
         url={https://doi.org/10.1006/jath.1994.1080},
      review={\MR{1292963}},
}

\bib{Ran84}{article}{
      author={Ransford, T.~J.},
       title={A short elementary proof of the {B}ishop-{S}tone-{W}eierstrass
  theorem},
        date={1984},
        ISSN={0305-0041},
     journal={Math. Proc. Cambridge Philos. Soc.},
      volume={96},
      number={2},
       pages={309\ndash 311},
         url={https://doi.org/10.1017/S0305004100062204},
      review={\MR{757664}},
}

\bib{Rud87}{book}{
      author={Rudin, Walter},
       title={Real and complex analysis},
     edition={Third},
   publisher={McGraw-Hill Book Co., New York},
        date={1987},
        ISBN={0-07-054234-1},
      review={\MR{924157}},
}

\bib{Rud91}{book}{
      author={Rudin, Walter},
       title={Functional analysis},
     edition={Second},
      series={International Series in Pure and Applied Mathematics},
   publisher={McGraw-Hill, Inc., New York},
        date={1991},
        ISBN={0-07-054236-8},
      review={\MR{1157815}},
}

\bib{Sum71}{article}{
      author={Summers, W.~H.},
       title={The general complex bounded case of the strict weighted
  approximation problem},
        date={1971},
        ISSN={0025-5831},
     journal={Math. Ann.},
      volume={192},
       pages={90\ndash 98},
         url={https://doi.org/10.1007/BF02052753},
      review={\MR{0284800}},
}

\bib{Tim05}{article}{
      author={Timofte, Vlad},
       title={Stone-{W}eierstrass theorems revisited},
        date={2005},
        ISSN={0021-9045},
     journal={J. Approx. Theory},
      volume={136},
      number={1},
       pages={45\ndash 59},
         url={https://doi.org/10.1016/j.jat.2005.05.004},
      review={\MR{2165118}},
}

\end{biblist}
\end{bibdiv}

\end{document}